%% file: main.tex
\documentclass{article}
\input{shared}

\title{Exponential convergence of Sobolev gradient descent for a class of nonlinear eigenproblems}
\date{}
\author{Ziyun Zhang\thanks{Applied and Computational Mathematics, Caltech, Pasadena, CA 91125 ({zyzhang@caltech.edu}).}}
\begin{document}
\maketitle

\begin{abstract}
We propose to use the {\L}ojasiewicz inequality as a general tool for analyzing the convergence rate of gradient descent on a Hilbert manifold, without resorting to the continuous gradient flow. Using this tool, we show that a Sobolev gradient descent method with adaptive inner product converges exponentially fast to the ground state for the Gross-Pitaevskii eigenproblem. This method can be extended to a class of general high-degree optimizations or nonlinear eigenproblems under certain conditions. We demonstrate this generalization using several examples, in particular a nonlinear Schr\"odinger eigenproblem with an extra high-order interaction term. Numerical experiments are presented for these problems.

\end{abstract}


\section{Introduction}

The Gross-Pitaevskii eigenproblem, a well-known example of the nonlinear Schr\"odinger eigenproblem, seeks $\lambda \in \R$ and $v\in H_0^1(\Omega)$ that satisfy the following equation
\begin{equation}
\label{eq:eig_original}
    -\Delta v + V v + \beta |v|^2v = \lambda v \quad \text{on } \Omega \subset \R^d,
\end{equation}
where $\Omega$ is a bounded region in $\R^d$, $V(x)\ge 0$ is an external trapping potential, and $\beta \ge 0$ is a parameter describing the repulsive interaction between particles. In physics, this describes the Bose-Einstein condensate when the temperature is close to absolute zero. The eigenstate $v$ corresponding to the smallest $\lambda$ describes the ground state of this system. It has long been studied both in experiments \cite{anderson1995observation} and in numerical analysis \cite{cances2010numerical,dusson2017posteriori,henning2014two,lieb2001bosons}. 

To find the ground state $v$ is equivalent to solving the following minimization problem:
\begin{equation}
\label{eq:E}
    \min_{\|u\|_{L^2}=1,\,\, u\in H_0^1(\Omega)} E(u):= \int_{\Omega} \left( |\nabla u|^2+V|u|^2+\frac{\beta}{2} |u|^4 \right) \rd x.
\end{equation}
The constraint set $\{u \in H_0^1(\Omega): \|u\|_{L^2}=1\}$ is the unit sphere in $H_0^1(\Omega)$. It can be seen as an infinite dimensional Hilbert manifold. Such a manifold (with additional $L^\infty(\Omega)$ constraints) will be denoted as $\M$ in subsequent sections. Thus many manifold optimization methods on the Riemannian manifold are readily applicable to this problem, with diverse techniques and rich theories.

In this paper, we focus on a special manifold gradient descent method named the \emph{Sobolev projected gradient descent (Sobolev PGD)}, first proposed in \cite{henning2020sobolev}. This method has the following iteration formula:
\begin{align}
\label{eq:SPGD}
    u_{n+1} = R\left(\left(1-\tau_n\right)u_n + \tau_n \cdot \frac{(u_n,u_n)_{L^2}}{(\G_{u_n} u_n,\G_{u_n} u_n)_{a_{u_n}}}\G_{u_n} u_n\right),
\end{align}
where $R$ is the retraction back onto the manifold, $\tau_n$ is the $n$-th step size, $(\cdot, \cdot)_{a_{u_n}}$ is an adaptive inner product in the tangent space of $\M$, and $\Gun$ is \rev{the Greens operator associated with $(\cdot, \cdot)_{a_{u_n}}$}. Their definitions are in Section \ref{sec:main}. 
The main result of this paper is as follows.
\begin{theorem}[Main result, informal]
    If initialized with a positive initial guess $u_0$, the $a_u$-Sobolev gradient descent which is given by (\ref{eq:SPGD}) converges to the ground state of the eigenproblem (\ref{eq:eig_original}) exponentially fast.
\end{theorem}

The idea of using a discretized normalized gradient flow (DNGF) to solve Problem (\ref{eq:E}) can be traced back to \cite{bao2004computing}. Following this seminal work there have been a number of variants, see e.g. \cite{danaila2010new,danaila2017computation,faou2018convergence} and the review paper \cite{bao2013mathematical}. The viewpoint of (Riemannian) manifold optimization has also been explicitly adopted in \cite{danaila2017computation}. 
Based on those methods with fixed inner products, the adaptive version of $a_u$-Sobolev gradient descent has recently been proposed in \cite{henning2020sobolev}. 
Despite its popularity, quantitative convergence analysis of the DNGF family has been quite lacking. The convergence rate has been either unavailable, or only proved for the gradient flow \cite{henning2020sobolev}.
Another popular choice is the self consistent field iteration (SCF), see e.g. \cite{cances2000convergence}. Rigorous global convergence rate is however difficult to establish.
There are also second-order methods like the Riemannian Newton method, but they require second-order information which can be expensive to obtain.

\rev{We highlight the main differences between the current paper and \cite{henning2020sobolev}. The authors of \cite{henning2020sobolev} first propose the Sobolev gradient descent method (\ref{eq:SPGD}). They establish the exponential convergence rate of the \emph{time-continuous} gradient flow. But the important question of whether the \emph{time-discrete} gradient descent also achieves optimal exponential convergence rate remains open. Our main contribution is to give a confirmatory answer to this question. We do this by introducing the \emph{{\L}ojasiewicz inequality} tool, which is a general analytical tool that is applicable to a wide class of problems.


Specifically,} in Section \ref{sec:Lojasiewicz}, using the {{\L}ojasiewicz inequality} tool, we reveal that the key to exponential convergence is the quadratic nature of the objective energy functional. In other words, regarded as a polynomial, the objective functional should behave like a degree-2 polynomial under the given manifold metric. The {\L}ojasiewicz inequality has been widely used in the optimization community, see e.g. \cite{frankel2015splitting,schneider2015convergence}. Yet it has scarcely been applied to the problems of interest in this paper. 

Although the degree of polynomial of the objective function in Problem (\ref{eq:E}) is formally higher than quadratic, Method (\ref{eq:SPGD}) changes the situation by using an adaptive inner product $a_u(\cdot, \cdot)$ instead of a fixed inner product. As a comparison, using a fixed inner product, the {\L}ojasiewicz exponent (the $\theta$ in Theorem \ref{thm:LDS}) calculated in \cite{zhang2019geometric} is $1/4$; while in this paper, using an adaptive inner product, we have $\theta = 1/2$. The latter is more desirable according to Theorem \ref{thm:LDS}. Thus, in Section \ref{sec:main}, using the {\L}ojasiewicz inequality tool, we are able to prove the exponential convergence rate of discrete time gradient descent directly. 

The {\L}ojasiewicz inequality tool also makes the Sobolev gradient descent easily applicable to general optimization of high-degree objective or eigenvalue problems other than the Gross-Pitaevskii eigenvalue problem. Its interesting property of making a high-degree polynomial behave like quadratic is not specific to a certain problem, but is general. Examples include the biharmonic Schr\"odinger, the nonlinear Schr\"odinger with a different order or extra interaction terms, and potentially some general manifold optimization problems. 

In addition to the necessary regularity conditions, the only essential requirement is that the global ground state of the nonlinear problem is also the unique ground state of its \emph{linearized} version, what we call the ``double ground state'' property.\footnote{This property is nontrivial. Although an eigenstate of the nonlinear problem is always an eigenstate of the linearized problem, it is not always the \emph{lowest energy} eigenstate (i.e., ground state) of the linearized problem.} For Problem (\ref{eq:eig_original}), this property will be rigorously proved in Section \ref{sec:LDS_GP}. For many other problems, it is either provable, or a reasonable assumption according to numerical evidence. We summarize this result as the following:

\begin{proposition}[Generalization of main result, informal]
    If the objective problem satisfies the ``double ground state'' property and necessary regularity conditions, then with a proper initialization $u_0$, the $a_u$-Sobolev gradient descent converges to a minimizer of this problem exponentially fast.
\end{proposition}

Specifically, an example of nonlinear Schr\"odinger eigenproblem from \cite{bao2019ground} will be rigorously discussed in Section \ref{sec:gen}. This example has an extra high-order interaction term $- \delta \Delta (|v|^2)v$ where $\delta \ge 0$. Classical methods that work for (\ref{eq:eig_original}) could become inefficient or unstable for this problem. A density function reformulation $\rho := |u|^2$ was proposed in \cite{bao2019computing}, but it has to treat the lack of continuity of $\nabla \sqrt{\rho}$ near $0^+$ with extra regularization. Therefore the adaptive Sobolev gradient descent is advantageous for its simplicity and fast convergence. 

We remark that if the domain is convex, an alternative approach to derive local linear convergence rate\footnote{Both \emph{exponential} and \emph{linear convergence} refer to the case where $err_{k} \le c^k \cdot err_0$ for some $0<c<1$. In this paper we use both terms interchangeably. The term \emph{linear convergence} is more popular in the optimization community.} of gradient descent methods is to use \emph{strong convexity} (SC). This is especially popular in the finite dimensional data science problems \cite{chi2019nonconvex}. Attempts have also been made to extend it to nonconvex settings like manifolds. Some works in this direction can be found in \cite{absil2009optimization,cances2020convergence}. We emphasize that our approach using the \L ojasiewicz inequality has its advantages over SC, namely it applies to degenerate critical points where SC could fail, and it allows more freedom in the choice of iterative algorithms and convergence measures. A more detailed comparison of these two approaches would be of interest in future research.


The rest of the paper is organized as follows. In Section \ref{sec:Lojasiewicz}, we introduce the {\L}ojasiewicz inequality tool with mixed norms on the Hilbert manifold as an abstract convergence theorem. In Section \ref{sec:main}, we establish the main result on the exponential convergence of the $a_u$-Sobolev gradient descent method applied to the Gross-Pitaevskii eigenproblem (\ref{eq:eig_original}). Section \ref{sec:discretization} is devoted to the analysis of spatial discretization. In Section \ref{sec:gen}, we introduce several extensions of the Sobolev gradient descent to other nonlinear eigenproblems. Some numerical results are presented in Section \ref{sec:numerical}. Finally, we make some concluding remarks in Section \ref{sec:conclusion}.

\section{Abstract convergence theorem using the {\L}ojasiewicz inequality}
\label{sec:Lojasiewicz}

In this section, we introduce the {\L}ojasiewicz inequality tool as an abstract convergence theorem. We show that one can deduce the convergence of an iteration algorithm from a triplet of conditions (\ref{eq:L}), (\ref{eq:D}) and (\ref{eq:S}). Furthermore, whether the convergence rate is exponential (linear) or polynomial (sublinear) is determined by the exponent in the (\ref{eq:L}) inequality.

\begin{theorem}
\label{thm:LDS}
Assume that the domain $\M$ is a Hilbert manifold. \zzy{Let $\|\cdot\|_X$ be a norm on $\T\M$, the tangent bundle of $\M$, and $\|\cdot\|_Y$ be a norm in the ambient space of $\M$ which is complete. Here $\|\cdot\|_X$ and $\|\cdot\|_Y$ can be either same or different.} Let $\{u_n\}_{n=0}^{\infty} \subset \M$ be a sequence generated by some iterative algorithm. \zzy{Assume that $E(u)$ is differentiable on $\M$ and} let $\text{grad } E(u)$ be the manifold gradient of $E(u)$. If $E(u)$ and $\{u_n\}_{n=0}^{\infty}$ satisfy the following conditions for all $n\in\mathbb{Z}_+$:
\begin{itemize}
    \item \emph{({\L}ojasiewicz Gradient Inequality)} 
    There exists $u^*$ that is a cluster point of $\{u_n\}$, and there exists $0<C_L<+\infty$, $0<\theta\leq \frac{1}{2}$, such that for large enough $n$,
    \begin{equation}
    \label{eq:L}\tag{L}
        \left|E(u_n)-E(u^*)\right|^{1-\theta}\leq C_L \|\text{grad } E(u_n)\|_X ;   
    \end{equation}
    \item \emph{(Descent Inequality)} There exists $C_D>0$ such that for large enough $n$,
    \begin{equation}
    \label{eq:D}\tag{D}
        E(u_n)-E(u_{n+1}) \geq C_D \|\text{grad } E(u_n)\|_X \|u_{n+1}-u_n\|_Y;
    \end{equation}
    \item \emph{(Step-size Condition)} There exists $C_S>0$ such that for large enough $n$,
    \begin{equation}
    \label{eq:S}\tag{S}
        \|u_{n+1}-u_n\|_Y \ge C_S \|\text{grad } E(u_n)\|_X.
    \end{equation}
\end{itemize}
Then $u^*$ is the \zzy{unique} limit point of $\{u_n\}_{n=0}^{\infty}$ w.r.t. $\|\cdot\|_Y$. Moreover, $\{u_n\}_{n=0}^{\infty}$ converge to $u^*$ with the following \zzy{asymptotic} convergence rate:
\begin{align*}
    \|u_n-u^*\|_Y \lesssim
    \begin{cases}
        e^{-cn}, & \text{if } \theta = \frac{1}{2}, \\
        n^{-\frac{\theta}{1-2\theta}}, & \text{if } \theta \in (0,\frac{1}{2}),
    \end{cases}
\end{align*}
where $c := \text{log}\,(1-\frac{C_DC_S}{2C_L^2})$.
\end{theorem}
\begin{proof}
    $\{E(u_n)\}$ is monotonically decreasing from Condition (\ref{eq:D}). Since $u^*$ is a cluster point of $\{u_n\}$, $E(u_n) \ge E(u^*)$ for any $n$. We also have $\lim_{n\to\infty}E(u_n) = E(u^*)$ by continuity of $E(\cdot)$. Without loss of generality, assume that $E(u^*)=0$. By Conditions (\ref{eq:D}) and (\ref{eq:L}), we have
    \begin{align*}
        \|u_{n+1}-u_n\|_Y &\le \frac{E(u_n)-E(u_{n+1})}{C_D \|\text{grad } E(u_n)\|_X} \le \frac{C_L}{C_D}(E(u_n)-E(u_{n+1}))E(u_n)^{\theta-1} \\ 
        &\le \frac{C_L}{C_D} \int_{E(u_{n+1})}^{E(u_n)} y^{\theta-1} \,\rd y = \frac{C_L}{\theta C_D} (E(u_n)^\theta-E(u_{n+1})^\theta).
    \end{align*}
    Using a bootstrapping argument, we have that for any $m>n$, 
    \begin{align}
    \label{eq:proof1}
        \|u_n-u_m\|_Y \le \frac{C_L}{\theta C_D}(E(u_n)^\theta-E(u_{m})^\theta) \le \frac{C_L}{\theta C_D}E(u_n)^\theta.
    \end{align}
    Since $E(u_n)$ is convergent, we deduce that $u_n$ is convergent, and the limit point is $u^*$. 
    
    To estimate the convergence rate, let $r_n:= \sum_{k=n}^{\infty}\|u_{k+1}-u_k\|_Y$, then $\|u_n-u^*\|_Y\le r_n$. It suffices to estimate the convergence rate of $r_n$. By Conditions (\ref{eq:L}) and (\ref{eq:S}), \zzy{for large enough $n$,}
    \begin{align*}
        \left|E(u_n)-E(u^*)\right|^{1-\theta}\leq C_L \|\text{grad } E(u_n)\|_X  \le \frac{C_L}{C_S} \|u_{n+1}-u_n\|_Y.
    \end{align*}
    Since we have made the assumption that $E(u^*) =0$, we obtain
    \begin{align}
    \label{eq:proof2}
        E(u_n) \le \left(\frac{C_L}{C_S} \|u_{n+1}-u_n\|_Y\right)^{\frac{1}{1-\theta}}.
    \end{align}
    Thus, we have
    \begin{align*}
        r_n & = \sum_{k=n}^{\infty}\|u_{k+1}-u_k\|_Y \le \sum_{k=n}^{\infty} \frac{C_L}{\theta C_D} (E(u_k)^\theta-E(u_{k+1})^\theta) = \frac{C_L}{\theta C_D}E(u_n)^\theta \\
        & \le \frac{C_L}{\theta C_D} \left(\frac{C_L}{C_S} \|u_{n+1}-u_n\|_Y\right)^{\frac{\theta}{1-\theta}} = \frac{C_L}{\theta C_D} \left(\frac{C_L}{C_S} (r_n-r_{n+1})\right)^{\frac{\theta}{1-\theta}},
    \end{align*}
    where the first inequality is due to (\ref{eq:proof1}) and the second inequality is due to (\ref{eq:proof2}). This gives
    \begin{align*}
        r_{n+1} \le r_n - C r_n^{\frac{1-\theta}{\theta}}, \quad C := C_L^{-\frac{1}{\theta}}(\theta C_D)^{\frac{1-\theta}{\theta}}C_S.
    \end{align*}
    Note that here $0<C<1$, otherwise the sequence would have converged in finite steps.
    
    If $\theta \in (0,\frac{1}{2})$, let $s_n := s_0 n^{-\gamma}$, $\gamma = \frac{\theta}{1-2\theta}$, and $s_0\ge\text{max}\{r_0, (C/\gamma)^{-\gamma}\}$. Then
    \begin{align*}
        s_{n+1} = s_n\left(1+\frac{1}{n}\right)^{-\gamma} \ge s_n\left(1-\frac{1}{n}\cdot \gamma\right) = s_n \left(1 - \gamma s_0^{-1/\gamma} s_n^{1/\gamma}\right) \ge s_n - C s_n^{\frac{\gamma+1}{\gamma}} = s_n - Cs_n^{\frac{1-\theta}{\theta}}.
    \end{align*}
    Combining $s_0\ge r_0$, $r_{n+1} \le r_n - C r_n^{\frac{1-\theta}{\theta}}$, and $ s_{n+1} \ge s_n - Cs_n^{\frac{1-\theta}{\theta}}$, by induction,
    \begin{align*}
        r_n \le s_n = s_0 n^{-\frac{\theta}{1-2\theta}} \quad \forall n,
    \end{align*}
    which is polynomial (or sub-linear) convergence.
    
    If $\theta = \frac{1}{2}$, then $r_{n+1} \le (1-C)r_n$, and 
    \begin{align*}
        r_n \le r_0 e^{c n}, \quad c:= \text{ln}(1-C),
    \end{align*}
    which is exponential (or linear) convergence.
\end{proof}

The above result can be seen as a generalization of Theorem 2.3 in \cite{schneider2015convergence} to the Hilbert space/manifold. Another work in this direction is \cite{frankel2015splitting}. What is new in our version is that one has the freedom to choose mixed norms ($\|\cdot\|_X$ and $\|\cdot\|_Y$), as long as the conditions (\ref{eq:L}), (\ref{eq:D}) and (\ref{eq:S}) can be satisfied under these norms. One example is the $\|\cdot\|_{a_u}$ in this paper, which varies with $u$. 

The advantage of the {\L}ojasiewicz inequality approach is that instead of dealing with the time discretization of the gradient flow, it gives the convergence of the gradient descent directly. The triplet of conditions (\ref{eq:L}), (\ref{eq:D}) and (\ref{eq:S}) in Theorem \ref{thm:LDS} all have clear and intuitive meanings. In fact, it is easier to deduce the convergence property of the gradient flow from that of the gradient descent, since we only need to take the limit $\tau \to 0^+$; while the reverse direction from gradient flow to gradient descent can be more difficult.

    An important observation is that the exponent $\theta$ in {\L}ojasiewicz gradient inequality indicates the \emph{degree of polynomial} of the objective function. For example, consider $x\in \R$, let $f(x) = x^k$ for a positive integer $k$, then {\L}ojasiewicz gradient inequality holds with $\theta = 1/k$. From this viewpoint, exponential convergence is closely related to certain quadratic-like behavior of the objective functional. It is thus unusual for a quartic-quadratic functional $E(\cdot)$ (i.e. a functional which is the sum of nonnegative quartic and quadratic terms) to have exponential convergence rate. What the Sobolev gradient does is to force the quartic term to behave like quadratic. This is the idea behind the proof of Theorem \ref{thm:GP-L}.

\section{Exponential convergence of Sobolev gradient descent}
\label{sec:main}

In this section, we establish the convergence rate of the $a_u$-Sobolev gradient descent for Problems (\ref{eq:eig_original}) and (\ref{eq:E}). In Section \ref{sec:BEC_intro}, we introduce the setting of manifold optimization and derive the $a_u$-Sobolev gradient descent method. In Section \ref{sec:LDS_GP}, using the {\L}ojasiewicz inequality tool from the previous section, we prove the exponential convergence rate by checking conditions (\ref{eq:L}), (\ref{eq:D}) and (\ref{eq:S}) for this specific method. 

\subsection{Manifold setting and derivation of \texorpdfstring{$a_u$}{Au}-Sobolev gradient descent.}
\label{sec:BEC_intro}

The following assumptions on $\Omega$, $V$ and $\beta$ will be required throughout this section.

\begin{assumption}
\label{ass:GP}
    Let $\Omega$, $V$ and $\beta$ be chosen such that the following assumptions hold:
    \begin{itemize}
        \item $\Omega$ is a bounded domain in $\R^d$, $d=1,\,2$, or $3$, and $\Omega$ is either convex Lipschitz or has a smooth boundary;
        \item $V\ge 0$ and $V\in L^\infty(\Omega)$, $V$ is a trapping potential, and $\beta\ge 0$. 
        
    \end{itemize}
\end{assumption}

\begin{remark}
    $V$ is chosen as a trapping potential so that the eigenstates of interest are localized. It is then natural to impose zero Dirichlet boundary conditions on $\partial \Omega$. Examples of a trapping potential include the well model in the classical Anderson localization where $\lim_{|x|\to\infty}V(x)=+\infty$, and the fully disordered model with high contrast and small interaction length.
\end{remark}

Define the infinite dimensional Hilbert manifold $\M$ as 
\begin{align*}
    \M := \{u\in H_0^1(\Omega) : \|u\|_{L^2(\Omega)} = 1, \, \zzy{\|u\|_{L^{\infty}(\Omega)} \le M_0} \text{ for some global constant } M_0\}.
\end{align*}
Then $\M$ is a submanifold in $H_0^1(\Omega) \cap \zzy{L^{\infty}(\Omega)}$. Note that although the original problem (\ref{eq:eig_original}) allows $v(x)\in\C$, we restrict our search to $u(x)\in \R$, as we will see that the existence of a real and positive ground state is ensured by Theorem \ref{thm:ground}. We also remark that $\|u\|_{L^{\infty}(\Omega)} \le M_0$ \rev{is not directly guaranteed by the iterative algorithm, but is rather left as an assumption.} It is a plausible assumption because we will see that the ground state $v$ is in ${L^{\infty}(\Omega)}$ by H\"older 
continuity in Theorem \ref{thm:ground}. 

For simplicity we drop $\Omega$ in norm and inner product notations when there is no confusion. 
The \emph{tangent space} of $\M$ at point $u\in\M$ is defined as
\begin{align}
\label{eq:tangent}
    \zzy{\mathcal{T}_u\M = \{\xi \in H_0^1(\Omega)\cap L^\infty(\Omega):\, (\xi, u)_{L^2} = 0\}.}
\end{align}
\zzy{We need an inner product in the tangent space, denoted as $(\cdot, \cdot)_X$. On the finite dimensional Riemannian manifold, this is dubbed the \emph{Riemannian metric}. It can be easily generalized to the infinite dimensional Hilbert manifold.}

For $u\ne 0$, the \emph{retraction} of $u$ onto $\M$ is given by
\begin{align*}
    R(u) = {u}/{\|u\|_{L^2}}.
\end{align*}
Note that the retraction operation itself is independent of the choice of the \zzy{inner product} $(\cdot, \cdot)_X$, but its approximation property is not. When the \zzy{inner product} $(\cdot, \cdot)_X$ is introduced, it is usually required that the retraction is at least first-order, i.e., $R(z+\xi) = z + o(\|\xi\|_X)$ for $z\in\M$ and $\xi \in \mathcal{T}_u\M$.

Given an inner product $(\cdot, \cdot)_X$, let $\G$ be its \rev{associated Greens operator}, i.e.,
\begin{align*}
    (z, \G w)_X = (z, w)_{L^2}, \qquad \forall \, z, w \in X.
\end{align*}
For an arbitrary element $\xi$ in the ambient space, the \emph{projection onto the tangent space} at point $u \in \M$ is given by
\begin{align*}
    P_{\mathcal{T}_u\M}(\xi) = \xi - \frac{(\xi,u)_{L^2}}{(\G u,\G u)_X}\G u.
\end{align*}

Given a differentiable function $E(u)$ defined on $\M$, the \emph{Sobolev gradient} of $E(u)$ with respect to the \zzy{inner product} $(\cdot, \cdot)_X$ is the unique element $\nabla_X E(u) \in X$ such that
\begin{align*}
    (\nabla_X E(u), w)_X = (\nabla E(u), w)_{L^2}, \qquad \forall \, w\in X.
\end{align*}
The \emph{manifold gradient} of $E(u)$ on $\M$, denoted as $\text{grad}E(u)$, is the projection of the Sobolev gradient onto the tangent space with respect to the \zzy{inner product} $(\cdot, \cdot)_X$. Thus we have
\begin{align*}
    \text{grad }E(u) = P_{\mathcal{T}_u\M}(\nabla_X E(u)) = \nabla_X E(u) - \frac{(\nabla_X E(u),u)_{L^2}}{(\G u,\G u)_X}\G u.
\end{align*}
It can be inferred from the above expression that $\text{grad }E(u) = 0$ implies $\nabla E(u) = \lambda u$ for some scalar $\lambda$. If $E(u)$ is as in (\ref{eq:E}), then $u$ is an eigenstate of (\ref{eq:eig_original}). This fact is independent of the choice of \zzy{inner product} $(\cdot, \cdot)_X$.

The choice of the \zzy{inner product in the tangent space} plays an important role in the analysis of manifold optimization algorithms as different \zzy{inner products} give different forms of gradient flow and gradient descent algorithms. Popular choices include $L^2$, $H^1$, and the $a_0$ inner product defined as follows:
\begin{align*}
    (z, w)_{a_0} := \int_{\Omega} \nabla z \nabla w + Vzw, \qquad \forall\, z,\,w \in \T_u\M,\quad u\in\M.
\end{align*}
All the above \zzy{inner products} are fixed everywhere on the manifold. Things become interesting when the inner product becomes \rev{adapted} to $u$. Specifically, we are interested in the following inner product
\begin{align}
    (z,w)_{a_u} := \int_{\Omega} \nabla z \nabla w + V z w + \beta |u|^2 z w, \qquad \forall\, z,\,w \in \T_u\M, \quad u\in\M,
\end{align}
and we define
\begin{align}
    \qquad \A_u := -\Delta + V + \beta|u|^2,
\end{align} 
such that $(\A_u z, w)_{L^2} = (z,w)_{a_u}$ for any  $z$, $w$.
This new \zzy{inner product} $(\cdot,\cdot)_{a_u}$ can be seen as the linearization of the Gross-Pitaevskii energy functional. A desirable property of this \zzy{inner product} is that the Sobolev gradient of $E(u)$ is $u$ itself, i.e.,
\begin{align}
    \nabla_{a_u} E(u) = u.
\end{align}
This inner product has the \rev{associated Greens operator} $\G_u$ whose properties have been explored in \cite{henning2020sobolev}.

\begin{lemma}
    Under the adaptive \zzy{inner product} $(\cdot, \cdot)_{a_u}$, the retraction $R$ is second-order.
\end{lemma}
\begin{proof}
    \rev{For $u\in\M$ and for any $\xi \in \mathcal{T}_u\M$,
    \begin{align*}
        \frac{\|R(u+\xi)-(u+\xi)\|_{a_u}}{\|u+\xi\|_{a_u}} 
        = \frac{\|(1-1/\|u+\xi\|_{L^2})(u+\xi)\|_{a_u}}{\|u+\xi\|_{a_u}} = \left|1- \frac{1}{\|u+\xi\|_{L^2}}\right|.
    \end{align*}
    Note that $\xi$ is a tangent vector of the manifold at $u$. By (\ref{eq:tangent}), $\|u+\xi\|_{L^2}^2 = \|u\|_{L^2}^2 + \|\xi\|_{L^2}^2 + 2(\xi, u)_{L^2} = 1+\|\xi\|_{L^2}^2$. Thus we have
    \begin{align*}
        \frac{\|R(u+\xi)-(u+\xi)\|_{a_u}}{\|u+\xi\|_{a_u}}= \left|1- (1+\|\xi\|_{L^2}^2)^{-1/2}\right| = \frac{1}{2}\|\xi\|_{L^2}^2 + \mathcal{O}(\|\xi\|_{L^2}^4).
    \end{align*}    }
    By the Poincar\'e inequality, when $V\ge 0$ and $\beta\ge 0$,
    \begin{align*}
        \|\xi\|_{L^2}^2 \le C_P \|\nabla \xi\|_{L^2}^2 \le C_P \|\xi\|_{a_u}^2
    \end{align*}
    for some domain constant $C_P>0$.
    Thus we have
    \begin{align*}
        \|R(u+\xi)-(u+\xi)\|_{a_u} = \mathcal{O}(\|\xi\|_{a_u}^2),
    \end{align*}
    where the constant in $\O(\cdot)$ is independent of $\xi$.
\end{proof}

Using the \zzy{inner product} $(\cdot, \cdot)_{a_u}$, the manifold gradient becomes
\begin{align}
\label{eq:grad}
    \text{grad } E(u) = u - \frac{(u,\,u)_{L^2}}{(\G_{u}u,\,\G_{u} u)_{a_{u}}}\G_{u} u.
\end{align}
We now have the Sobolev projected gradient descent (Sobolev PGD) as in (\ref{eq:SPGD}):
\begin{align}
\begin{split}
    u_{n+1} &= R\left(u_n - \tau_n\cdot \text{grad }E(u_n)\right) \\
    &= R\left(\left(1-\tau_n\right)u_n + \tau_n \cdot \frac{(u_n,u_n)_{L^2}}{(\G_{u_n} u_n,\G_{u_n} u_n)_{a_{u_n}}}\G_{u_n} u_n\right).
\end{split}
\end{align}


\subsection{Asymptotic convergence and exponential rate.}
\label{sec:LDS_GP}

Throughout the rest of the paper, let $v$ always denote the global minimizer of $E(u)$, i.e. the ground state of the nonlinear eigenproblem. Let $\lambda$ always denote its corresponding eigenvalue. We have the following basic observations about the ground state $v$.

\begin{theorem}
\label{thm:ground}
    There is a ground state $v$ that satisfies $v(x)>0$ everywhere on $\Omega$. It is the only strictly positive eigenstate of (\ref{eq:eig_original}) up to scaling. Moreover, it is both the \emph{unique} ground state of the nonlinear eigenproblem (\ref{eq:eig_original}) and the \emph{unique} ground state of the linearized operator $\mathcal{A}_v$ up to the sign. Moreover, $v$ has H\"older regularity $v\in C^{0,\alpha}(\bar{\Omega})$ for some $0<\alpha<1$.
\end{theorem}
\begin{proof}
    This theorem is a consequence of Lemma 2 in \cite{cances2010numerical} and Lemmas 5.3 and 5.4 in \cite{henning2020sobolev}. We only outline the main idea of the proof here to make this paper self-contained.
    
    The idea is that the existence of at least one global minimizer $v$ is ensured by the convexity of $E(u)$. The H\"older continuity of $v$ is ensured by elliptic regularity, see e.g. \cite[Theorem~8.24]{gilbarg2015elliptic}. This $v$ can always be chosen to be nonnegative because $E(u) = E(|u|)$. This nonnegativity can be made into positivity by applying the Harnack inequality to $(A_v-\lambda)$, see e.g. \cite[Corollary~8.21]{gilbarg2015elliptic}. Thus, there exists a ground state of the nonlinear problem that is positive. The same argument shows that the ground state eigenfunction of the linearized operator $\A_v$ is also positive and is unique. Since $v$ is an eigenfunction of $\A_v$ and is positive, it is exactly that ground state. Thus we have the ``double ground state'' property. Finally, the uniqueness of any positive eigenstate of the original nonlinear eigenproblem can be established by contradiction. This can be done either by the Picone identity as in \cite{henning2020sobolev}, or by showing that as long as some $u$ itself is the ground state of the linearized operator $\A_u$, it must be the ground state of the original problem.
\end{proof}

It turns out in subsequent results that $v$ being the ``double'' ground state in Theorem \ref{thm:ground} is essential to the exponential convergence rate. 

\begin{lemma}
\label{lemma:nonnegative}
    If the initial point $u_0$ of the Sobolev PGD satisfies $u_0>0$ everywhere on $\Omega$, then $\{u_n\}_{n=0}^\infty$ generated by the Sobolev PGD with step size $\tau_{\text{min}}\le\tau_n\le \tau_{\text{max}}$ \zzy{for some $0<\tau_{min}\le \tau_{\text{max}}\le 1$} converges to the ground state $v$ strongly in $H^1(\Omega)$. 
\end{lemma}
\begin{proof}
    \zzy{The proof is originally developed in \cite{henning2020sobolev} and we only outline its main idea here to make this paper self-contained. The key idea is to show that $u_n(x) \ge 0$ for all $n$ by induction.} Assume that $u_n\ge 0$, we will show that this implies $\mathcal{G}_{u_n}u_n\ge 0$, and with $\tau_n\le 1$ this implies $u_{n+1}\ge 0$. 
    
    Specifically, observe that $\mathcal{G}_{u_n} u_n$ is the unique minimizer of 
    \begin{align*}
        \phi(y) := (y,y)_{a_{u_n}} - 2(y,u_n)_{L^2}.
    \end{align*}
    Since $u_n\ge 0$, we have that $\phi(|y|)\le \phi(y)$ $\forall y$. This implies that the minimizer of $\phi(\cdot)$ is nonnegative because we can always take the absolute value of the variable without increasing the functional value. Thus, $\mathcal{G}_{u_n}u_n\ge 0$. We then use the fact that $u_{n+1}$ is the scaled weighted average of two nonnegative quantities:
    \begin{align*}
        \tilde{u}_{n+1} = (1-\tau_n)u_n+\tau_n\gamma_n\mathcal{G}_{u_n}u_n,\quad \gamma_n =\frac{(u_n,u_n)_{L^2}}{(\mathcal{G}_{u_n}u_n, \mathcal{G}_{u_n}u_n)_{a_{u_n}}}\ge 0, \quad u_{n+1} = \tilde{u}_{n+1}/\|\tilde{u}_{n+1}\|_{L^2}.
    \end{align*}
    Thus, we establish that $u_n\ge 0$ implies $u_{n+1}\ge 0$. Since $u_0>0$, we have that $u_n\ge 0$ for all $n$.
    
    The existence of a cluster point $u^*$ for $\{u_n\}$ can be ensured by energy decay. 
    This convergence to $u^*$ is in the sense of weak convergence in $H_0^1(\Omega)$. From the above induction, $u^*$ is nonnegative, and following an argument similar to that in Theorem \ref{thm:ground} we can show that it is all positive.
    
    Since the step size is lower-bounded, $u^*$ must be a fixed point of $E(u)$, where $\text{grad } E(u^*)=0$. As we mentioned above, $\text{grad }E(u^*) = 0$ implies $\nabla E(u^*) = \lambda u^*$ for some scalar $\lambda$, i.e., $u^*$ is an eigenstate of the eigenvalue problem (\ref{eq:eig_original}).
    From the uniqueness result of positive eigenstate in Theorem \ref{thm:ground} we know that it could only be the ground state $v$. Therefore, $\{u_n\}$ converges to $v$ itself. 
    
    Finally, the weak convergence in $H_0^1(\Omega)$ implies strong convergence in $L^p(\Omega)$ for $p<6$ by the Rellich-Kondrachov embedding. This would give the convergence of energy $\{E(u_n)\}$, and consequently strong convergence in $H^1(\Omega)$.
\end{proof}

Before proceeding to the proof of Conditions (\ref{eq:L}), (\ref{eq:D}) and (\ref{eq:S}), we first need some technical lemmas.

\begin{lemma}[Norm equivalence]
\label{lemma:equiv}
    Under Assumptions \ref{ass:GP}, there exist positive constants $C_E$, $\widetilde{C_E}$ depending only on $\beta$, $M_0$, $V$, and the domain $\Omega$, such that 
    \begin{align*}
        &C_E \|\cdot\|_{a_u} \le  \|\cdot\|_{a_0} \le  C_E^{-1}\|\cdot\|_{a_u},\\
        &\widetilde{C_E} \|\cdot\|_{a_u} \le  \|\cdot\|_{H^1} \le  \widetilde{C_E}^{-1}\|\cdot\|_{a_u}.
    \end{align*}
\end{lemma}

\begin{proof}
    See Appendix \ref{sec:proof_lemma:equiv}.
\end{proof}

In the next two lemmas, let $\lambda_i$ and $\mu_i$ be the $i$-th smallest eigenvalues of $\A_v$ and $\A_u$ respectively, and $v_i$ and $w_i$ be their corresponding eigenfunctions satisfying $\|v_i\|_{L^2} = 1$ and $\|w_i\|_{L^2}$ (so that $v = v_1$, $\lambda = \lambda_1$). Theorem \ref{thm:ground} has ensured the uniqueness of the ground state. The fact that $\A_v$ only has point spectrum ensures that there is a positive gap $C_v$ between $\lambda_1$ and $\lambda_2$.

\zzy{
\begin{lemma}[Perturbation of eigenvalues and eigenfunctions]
\label{lemma:perturb}   
    Under Assumptions \ref{ass:GP}, there exists a positive constant $C = C(\beta,V,M_0,\Omega,\lambda_1,C_v)$, such that \rev{for all $u\in\M$ satisfying $\|u-v\|_{H^1}\le C$}, we have that $\|u-w_1\|_{L^2} \le s$ for some $s<1$.
\end{lemma}
}

\begin{proof}
    See Appendix \ref{sec:proof_lemma:perturb}.
\end{proof}

\begin{lemma}[Condition (\ref{eq:L}) for the linearized operator]
\label{lemma:linear}
    Let $\A: X \to X$ be a symmetric and positive definite linear operator on the Hilbert space with a bounded Greens operator $\G$. 
    Let $\mu_i$ denote the $i$-th smallest eigenvalue of $\A$, and $w_i$ be its corresponding (normalized) eigenfunction. Assume that $\mu_2 > \mu_1$. Then for any $u$ such that $\|u\|_{L^2}=1$ and $\|u-w_1\|_{L^2} \le {s} <1$, we have 
    \begin{align*}
        (u,u)_\A - (w_1, w_1)_\A \le C_L \left((u,u)_\A - \frac{1}{(u,\G u)_{L^2}}\right)
    \end{align*}
    for some constant $C_L$ that depends only on ${s}$, $\mu_1$ and $\mu_2$.
\end{lemma}

\begin{proof}
    See Appendix \ref{sec:proof_lemma:linear}.
\end{proof}

Using the above technical lemmas, we are now ready to prove the following theorems. They show that the sequence $\{u_n\}$ generated by (\ref{eq:SPGD}) satisfies Conditions (\ref{eq:L}), (\ref{eq:D}) and (\ref{eq:S}).

The first theorem is on Condition (\ref{eq:L}) near the ground state $v$ of the nonlinear eigenproblem. It is the central one of the three theorems.
\begin{theorem}
\label{thm:GP-L} 
    Under Assumptions \ref{ass:GP}, Condition (\ref{eq:L}) is satisfied for $\|\cdot\|_X = \|\cdot\|_{a_u}$ and $\theta=\frac{1}{2}$ near the ground state $v$. In other words, there exists some constant $C>0$, such that for any $u$ in $\{u:\,u\in\M, \, E(u)\ge E(v),\, \|u-v\|_{H^1} \le C\}$, we have
    \begin{align*}
        \left|E(u)-E(v)\right|^{\frac{1}{2}} \leq C_L \|\text{grad } E(u)\|_{a_{u}}.
    \end{align*}
\end{theorem}
\begin{proof}
    First notice that for any $u$ in the constraint set of the theorem, $E(u)-E(v) \le a_u(u,u) - a_u(v,v)$. This is because
    \begin{align*}
        E(u)-E(v) - \left((u,u)_{a_u} - (v,v)_{a_u}\right) &= -\frac{\beta}{2}\int_{\Omega} u^4 -\frac{\beta}{2}\int_{\Omega} v^4 +\beta\int_{\Omega} u^2v^2 \\
        &= -\frac{\beta}{2} \int_{\Omega} (u^2-v^2)^2 \le 0.
    \end{align*}
    Let $w_1$ be the eigenfunction corresponding to the smallest eigenvalue of $\A_u$, then 
    \begin{align*}
        (u,u)_{a_u} - (v,v)_{a_u} \le (u,u)_{a_u} - (w_1,w_1)_{a_u}.
    \end{align*}
    On the other hand, by (\ref{eq:grad}), we have
    \begin{align*}
        \|\text{grad } E(u)\|_{a_u}^2 = \left\|u - \frac{(u,\,u)_{L^2}}{(\G_{u}u,\,\G_{u} u)_{a_{u}}}\G_{u} u\right\|_{a_u}^2 = \left\|u - \frac{\Gu u}{(u,\mathcal{G}_u u)_{L^2}}\right\|_{a_u}^2 = (u,u)_{a_u} - \frac{1}{(u,\Gu u)_{L^2}}.
    \end{align*}
    It suffices to show that
    \begin{align}
    \label{eq:proof3}
        (u,u)_{a_u} - (w_1,w_1)_{a_u} \le C_L \left((u,u)_{a_u} - \frac{1}{(u,\Gu u)_{L^2}}\right),
    \end{align}
    which only involves the inner product $(\cdot, \cdot)_{a_u}$. 
    
    Using Lemma \ref{lemma:perturb}, we have that there exists $C>0$  
    such that when $\|u-v\|_{H^1}<C$, we have $\|u-w_1\|_{L^2} \le s$ for some constant $s<1$. Thus, Lemma \ref{lemma:linear} is applicable to $(\cdot, \cdot)_{a_u}$. This gives the above inequality on $(\cdot, \cdot)_{a_u}$, with a constant $C_L$ depending only on $\beta, \, V,\, M_0,\, \Omega,\,\lambda_1$, and $C_v$. The {\L}ojasiewicz inequality can thus be achieved.
\end{proof}

\begin{remark}
    The above proof of Condition (\ref{eq:L}) depends crucially on Lemma \ref{lemma:linear}. Lemma \ref{lemma:linear} can be seen as the version of the {\L}ojasiewicz inequality with $\theta=\frac{1}{2}$ for a linear operator $\A$. So its primary consequence is the linear convergence rate of the proposed algorithm to the ground state of a linear operator $\A$.
    
    The key idea of the proof Theorem \ref{thm:GP-L}, then, is to reduce it to the inequality (\ref{eq:proof3}). The inequality (\ref{eq:proof3}) only involves the operator $\A_u$, which is bilinear. Although $\A_u$ formally depends on $u$, the inequality (\ref{eq:proof3}) itself is not affected by nonlinearity. So Lemma \ref{lemma:linear} can be applied to prove (\ref{eq:proof3}). 
    
    Thus, one way to interpret the proof of Theorem \ref{thm:GP-L} is to view it as linearizing the nonlinear eigenproblem (\ref{eq:eig_original}) using the adaptive inner product $(\cdot,\cdot)_{a_u}$, so that it preserves the  {\L}ojasiewicz property with  $\theta=\frac{1}{2}$.
\end{remark}

The next theorem is on Condition (\ref{eq:D}) for the sequence generated by the proposed algorithm. 

\begin{theorem}
\label{thm:GP-D}
    Under Assumptions \ref{ass:GP}, Condition (\ref{eq:D}) is satisfied for $\|\cdot\|_X = \|\cdot\|_{a_u}$, $\|\cdot\|_Y = \|\cdot\|_{a_0}$ if $\{u_n\}$ is generated by the Sobolev projected gradient descent with step size $0<\tau_n\le \tau_{\text{max}}$ \zzy{for some $\tau_{\text{max}}>0$}, i.e., 
    \begin{align*}
        E(u_n) - E(u_{n+1}) \ge C_D \|\text{grad }E(u_n)\|_{a_{u_n}}\|u_n-u_{n+1}\|_{a_0}.
    \end{align*}
\end{theorem}


\begin{proof}
    It is obvious that $\|u_n-u_{n+1}\|_{a_0} \le \|u_n-u_{n+1}\|_{a_{u_n}}$. Since $\{u_n\}$ is generated by the Sobolev projected gradient descent algorithm, we have
    \begin{align*}
        u_{n+1} &= R\left(u_n - \tau_n \cdot\text{grad }E(u_n)\right),\\
        \text{grad }E(u_n) &= u_n - \frac{(u_n,u_n)_{L^2}}{(\G_{u_n} u_n,\G_{u_n} u_n)_{a_{u_n}}}\G_{u_n} u_n = u_n- \frac{\mathcal{G}_{u_n} u_n}{(u_n,\mathcal{G}_{u_n} u_n)_{L^2}}.
    \end{align*}
    The second-order retraction property implies that
    \begin{align*}
        u_n - u_{n+1} = \tau_n \left(u_n- \frac{\mathcal{G}_{u_n} u_n}{(u_n,\mathcal{G}_{u_n} u_n)_{L^2}}\right) + \mathcal{O}(\tau_n^2).
    \end{align*}
    Thus, we obtain
    \begin{align*}
        E(u_n) - E(u_{n+1}) &= \left(u_n - u_{n+1}, \,\, \nabla_{a_{u_n}} E(u_n) \right)_{a_{u_n}} +\mathcal{O}(\|u_n - u_{n+1}\|^2) \\
        &= \left(u_n - u_{n+1}, \,\, u_n \right)_{a_{u_n}} +\mathcal{O}(\|u_n - u_{n+1}\|^2) \\
        &=  \tau_n \left(u_n- \frac{\mathcal{G}_{u_n} u_n}{(u_n,\mathcal{G}_{u_n} u_n)_{L^2}}, \,\, u_n \right)_{a_{u_n}}+\mathcal{O}(\tau_n^2) \\
        &= \tau_n \left((u_n, u_n)_{a_{u_n}} - \frac{1}{(u_n,\mathcal{G}_{u_n} u_n)_{L^2}}\right) + \mathcal{O}(\tau_n^2).
    \end{align*}
    On the other hand, we have
    \begin{align*}
        \|\text{grad }E(u_n)\|_{a_{u_n}} 
        = \left((u_n, u_n)_{a_{u_n}} - \frac{1}{(u_n,\mathcal{G}_{u_n} u_n)_{L^2}}\right)^{\frac{1}{2}},
    \end{align*}
    and 
    \begin{align*}
        \|u_n-u_{n+1}\|_{a_{u_n}} &= \tau_n  \left\|u_n- \frac{\mathcal{G}_{u_n} u_n}{(u_n,\mathcal{G}_{u_n} u_n)_{L^2}}\right\|_{a_{u_n}} + \mathcal{O}(\tau_n^2) \\
        &= \tau_n \left((u_n, u_n)_{a_{u_n}} - \frac{1}{(u_n,\mathcal{G}_{u_n} u_n)_{L^2}}\right)^{\frac{1}{2}} + \mathcal{O}(\tau_n^2).
    \end{align*}
    This implies that
    \begin{align*}
        \|\text{grad }E(u_n)\|_{a_{u_n}}\|u_n-u_{n+1}\|_{a_0} 
        \le \tau_n \left((u_n, u_n)_{a_{u_n}} - \frac{1}{(u_n,\mathcal{G}_{u_n} u_n)_{L^2}}\right) + \mathcal{O}(\tau_n^2).
    \end{align*}
    Therefore, there exists a $\tau_{\text{max}}>0$ such that when $\tau\le\tau_{\text{max}}$, there exists $C_D$ such that Condition (\ref{eq:D}) holds. This $C_D$ only depends on $\tau_{\text{max}}$, but is independent of $u_n$.
\end{proof}

Next, we have the theorem is on Condition (\ref{eq:S}) for the sequence generated by the proposed algorithm. 

\begin{theorem}
\label{thm:GP-S}
    Under Assumptions \ref{ass:GP}, Condition (\ref{eq:S}) is satisfied for for $\|\cdot\|_X = \|\cdot\|_{a_u}$, $\|\cdot\|_Y = \|\cdot\|_{a_0}$ if $\{u_n\}$ is generated by the Sobolev projected gradient descent with step size $0<\tau_{\text{min}}\le \tau_n\le \tau_{\text{max}}$ \zzy{for some $0<\tau_{\text{min}}\le\tau_{\text{max}}$}, i.e.,
    \begin{align*}
        \|u_{n+1}-u_n\|_{a_0} \ge C_S \|\text{grad } E(u_n)\|_{a_{u_n}}.
    \end{align*}
\end{theorem}
\begin{proof}
    By Lemma \ref{lemma:equiv}, we have $\|u_{n+1}-u_n\|_{a_0} \ge C_E  \|u_{n+1}-u_n\|_{a_{u_n}}$ for some constant $C_E$. Note that in the previous proof we have shown that
    \begin{align*}
        \|\text{grad }E(u_n)\|_{a_{u_n}} 
        = \left((u_n, u_n)_{a_{u_n}} - \frac{1}{(u_n,\mathcal{G}_{u_n} u_n)_{L^2}}\right)^{\frac{1}{2}}
    \end{align*}
    and 
    \begin{align*}
        \|u_n-u_{n+1}\|_{a_{u_n}} = \tau_n \left((u_n, u_n)_{a_{u_n}} - \frac{1}{(u_n,\mathcal{G}_{u_n} u_n)_{L^2}}\right)^{\frac{1}{2}} + \mathcal{O}(\tau_n^2).
    \end{align*}
    Therefore, when $\tau_{\text{min}}\le \tau_n\le \tau_{\text{max}}$ \zzy{for some $0<\tau_{\text{min}}\le\tau_{\text{max}}$}, there exists a constant $C_S$ depending only on $C_E$, $\tau_{\text{min}}$ and $\tau_{\text{max}}$, such that 
    \begin{align*}
        \|u_{n+1}-u_n\|_{a_0} \ge C_S \|\text{grad }E(u_n)\|_{a_{u_n}}.
    \end{align*}
\end{proof}

Finally, we deduce the following results on the exponential convergence.

\begin{theorem}[Convergence rate of Sobolev PGD]
\label{thm:conv}
    If the Sobolev projected gradient descent for $E(u)$ converges to the ground state $v$, \zzy{and the step size $\{\tau_n\}$ satisfies $0<\tau_{\text{min}}\le \tau_n \le \tau_{\text{max}}$}, then it converges in the $a_0$-norm \zzy{with an asymptotic exponential convergence rate}.
\end{theorem}
\begin{proof}
    The proof follows directly from Theorems \ref{thm:LDS}, \ref{thm:GP-L}, \ref{thm:GP-D} and \ref{thm:GP-S}.
\end{proof}

\begin{corollary}[Global convergence to ground state]
\label{col:global}
    If the initial state $u_0$ satisfies $u_0\ge 0$ everywhere on $\Omega$, \zzy{and the step size $\{\tau_n\}$ satisfies $0<\tau_{\text{min}}\le \tau_n \le \tau_{\text{max}}$}, then the Sobolev projected gradient descent for $E(u)$ converges in the $a_0$-norm to the unique ground state \zzy{with an asymptotic exponential convergence rate}.
\end{corollary}
\begin{proof}
   Since the initial state is nonnegative, Lemma \ref{lemma:nonnegative} ensures the strong convergence of $\{u_n\}$ to the ground state $v$ in $H_0^1(\Omega)$. By Theorem \ref{thm:conv}, the asymptotic convergence rate in the $a_0$-norm is exponentially fast.
\end{proof}

Note that since the domain $\Omega$ is bounded, this convergence rate in the $a_0$-norm implies the exponential convergence rate in the $H^1$ or $L^2$ norm. \rev{We also remark that the optimal step size with theoretical guarantee depends on the values $\tau_\text{min}$ and $\tau_\text{max}$, which in turn depend on some properties of the ground state that is not known beforehand, but some practical choices of $\tau$ are demonstrated in the numerical experiments in Section \ref{sec:numerical}.}

\section{Spatial discretization}
\label{sec:discretization}

To solve the eigenproblem numerically using the computational procedure in the previous sections, we need to discretize the problem in the spatial domain $\Omega$. Let $\Omega_h$ be a spatial discretization with grid size $h$. \rev{Note that we only require $\Omega_h$ to be a convergent discretization, i.e., the solution to the discrete problem converges to that of the continuous problem as $h\to 0^+$, and the following analysis applies to general discretization schemes.} The discretized problem can be written as
\begin{equation}
\label{eq:eig_discrete}
    \min_{\|u_h\|_{L_h^2}=1, \,\, u_h\in\R^N} E_h(u_h):=  \|u_h\|_{\Lap_h}^2 + \|u_h\|_{V_h}^2+ \frac{\beta}{2} \|u_h\|_{L^4_h}^4,
\end{equation}
where
\begin{align*}
    \|u_h\|_{\Lap_h}^2 = u_h^\top (-\Lap_h) u_h \cdot  h^d, \quad
    \|u_h\|_{V_h}^2 = \sum_{i=1}^N V_h(i)u_h(i)^2 h^d, \quad 
    \|u_h\|_{L^p_h}^p = \sum_{i=1}^N u_h(i)^p h^d.
\end{align*}
Here $N$ denotes the total number of grid points, $(i)$ is an indexing of the grid points, i.e., $u_h(i)$  is the $i$-th entry of the vectorized $u_h$, $d$ is the dimension of the physical space, and ${\Lap_h}$ is the discretized Laplacian. 
The linearized operator $\mathcal{A}_{u,h}$ now has a matrix representation in $\R^{N \times N}$:
\begin{align*}
    \mathcal{A}_{u,h} = -\Lap_h + \text{diag} \{V_h + \beta u_h^{[2]}\},
\end{align*}
where $u_h^{[2]}(i):=u_h(i)^2$, i.e., $u_h^{[2]}$ is the \rev{componentwise} squared vector of $u_h$. \zzy{The respective norm is defined as $\|y\|^2_{\A_{u,h}}:= y^\top \A_{u,h}y$.}
We have the following results. 

\begin{theorem}[Discrete version of Theorem \ref{thm:ground}]
\label{thm:ground_discretized}
    There is a ground state $v_h$ of the discretized problem that satisfies $v_h>0$ everywhere on $\Omega_h$. It is the unique positive eigenstate of (\ref{eq:eig_discrete}). Moreover, it is both the \emph{unique} ground state of the nonlinear eigenproblem  (\ref{eq:eig_discrete}) and the \emph{unique} ground state of the linearized operator $\mathcal{A}_{v,h}$ up to the sign.
\end{theorem}
\begin{proof}
    \zzy{The existence of the ground state follows from the compactness of the constraint set $\{u_h: \, u_h\in\R^N, \,\,  \|u_h\|_{L_h^2}=1\}$ and the boundedness of $E_h(u_h)$. Thus it suffices to prove its uniqueness and positivity.}
    The proofs for the continuous version, i.e., Lemma 2 in \cite{cances2000convergence} and Lemmas 5.3 and 5.4 in \cite{henning2020sobolev}, need to be slightly modified to suit the discrete case. This is because the Harnack inequality and the Picone identity are only valid for continuous functions, and we need to establish their discrete counterparts. 
    
    One way to do this is to look at the convergence of the discretized eigenvector to its continuous counterpart at the small grid size limit $h\to 0^+$, see e.g. \cite{kuttler1970finite}. This is always possible no matter what kind of discretization we use. We do not present the details here.
    
    Another way is to observe that the discretized Laplacian, $\Lap_h$, is an M-matrix\footnote{An M-matrix is a matrix with nonnegative diagonal entries and nonpositive off-diagonal entries, \rev{with eigenvalues whose real parts are nonnegative.}} under some typical discretizations. Examples include finite difference discretization, and some P1-finite element discretizations. When $\Lap_h$ is an M-matrix, the proof can be simplified and the small $h$ constraint can be released.
    In this case, the proof takes the following steps:
    \begin{enumerate}[label={(\arabic*)}]
        \item \emph{For any $\mathcal{A}_{u,h}$, its eigenvector corresponding to the smallest eigenvalue can be chosen to be all positive, and is unique up to the sign.}
    
    Since $-\Lap_h$ has positive diagonals and non-positive off-diagonals, so does $\A_{u,h}$. Let $y$ be the  ground state eigenvector of $\A_{u,h}$, then $|y|^\top\mathcal{A}_{u, h}|y| \le y^\top \mathcal{A}_{u, h}y$. This is because $\mathcal{A}_{u, h}(i,i) y(i)^2 = \mathcal{A}_{u, h}(i,i) |y(i)|^2 $ for any $1\le i \le N$, and $\mathcal{A}_{u, h} (i,j) y(i) y(j) \ge \mathcal{A}_{u, h} (i,j) \cdot$ $|y(i)| |y(j)|$ for any $i\neq j$. As $y$ is the ground state eigenvector, this implies $y=|y|$, i.e., $y$ is nonnegative.\\
    We now show that $y$ is all positive.  If this is not true, then $y$ has some positive and some zero entries. So we can always find a zero entry $y{(i)}$ that is spatially next to a nonzero one, say $y{(j)}$, i.e., $y(i)=0$, $y(j)>0$, and $-\Lap_h(i,j)<0$. Then
    \begin{align*}
        0 &= \lambda y{(i)} = (\mathcal{A}_{u, h}y){(i)} = (-\Lap_h y)(i) + V_h(i)y(i) + \beta y(i)^3 \\
        &= (-\Lap_{h}y){(i)} = \sum_k -\Lap_{h}{(i,k)} y{(k)} = \sum_{k\neq i} -\Lap_{h}{(i,k)} y{(k)} \le -\Lap_{h}{(i,j)} y{(j)} < 0,
    \end{align*}
    which is a contradiction. Thus $y$ is all positive and is unique up to the sign.
    
        \item \emph{If $u_h$ itself is the smallest eigenvector of $\mathcal{A}_{u_h,h}$, then it is also the unique global minimizer of $E_h(u)$.}
    
    For any other $w_h\neq \pm u_h$, we have
    \begin{align*}
        E_h(w_h) - E_h(u_h) &= \|w_h\|^2_{\mathcal{A}_{u,h}} - \|u_h\|^2_{\mathcal{A}_{u,h}} + \frac{\beta}{2} \sum_{i=1}^N\left((w_h^{(i)})^4 + (u_h^{(i)})^4 - 2 (w_h^{(i)})^2(u_h^{(i)})^2\right) h^d  \\
        &= \left(\|w_h\|^2_{\mathcal{A}_{u,h}} - \|u_h\|^2_{\mathcal{A}_{u,h}}\right) + \frac{\beta}{2} \sum_{i=1}^N \left((w_h^{(i)})^2-(u_h^{(i)})^2\right)^2  h^d > 0.
    \end{align*}
    Thus $u_h$ is the unique global minimizer of $E_h(u)$.
    
        \item \emph{There is a unique positive eigenstate of (\ref{eq:eig_discrete}), which is the ground state of (\ref{eq:eig_discrete}) and the ground state of the linearized operator.}
   
    Any positive iteration sequence stays positive with gradient descent iteration. 
    \zzy{The compactness of the constraint set} ensures the existence of a sub-sequential limit point $v_h$, which is nonnegative. The fact that $v_h$ is the minimizer of $E_h(u)$ implies that it is an eigenstate of $\mathcal{A}_{v, h}$. By Step (1), this eigenstate is all positive and is thus the smallest eigenstate of $\mathcal{A}_{v,h}$. By Step (2), it is also the unique global minimizer of $E_h(u)$. 
    \end{enumerate}
\end{proof}

\begin{theorem}[Discrete version of Theorem \ref{thm:conv}] 
\label{thm:conv_discretized}
If the Sobolev PGD for $E_h(u)$ converges to the ground state $v_h$, \zzy{and the step size $\{\tau_n\}$ satisfies $0<\tau_{\text{min}}\le \tau_n \le \tau_{\text{max}}$}, then it converges \zzy{with an asymptotic exponential convergence rate}.
\end{theorem}
\begin{proof}
    Theorem \ref{thm:ground_discretized} ensures that $v_h$ is still the ``double'' ground state of both $E_h(u)$ and $\mathcal{A}_{v_h, h}$. Thus, Theorems \ref{thm:GP-L}, \ref{thm:GP-D}, and \ref{thm:GP-S} can all be generalized to the discretized case in the same way. The exponential convergence rate follows from the master theorem \ref{thm:LDS}.
\end{proof}

\begin{corollary}[Discrete version of Corollary \ref{col:global}]
   If the initial state $u_0$ satisfies $u_0(i)\ge 0$ $\forall i$, \zzy{and the step size $\{\tau_n\}$ satisfies $0<\tau_{\text{min}}\le \tau_n \le \tau_{\text{max}}$}, then the Sobolev PGD  for $E_h(u)$ converges to the unique ground state $v_h$ \zzy{with an asymptotic exponential convergence rate}.
\end{corollary}
\begin{proof}
    The proof follows similarly from the nonnegativity and uniqueness results of Theorem \ref{thm:ground_discretized} and the exponential convergence result of Theorem \ref{thm:conv_discretized}.
\end{proof}


\section{Generalization to other nonlinear eigenproblems}
\label{sec:gen}

The Sobolev PGD points out a new direction for first-order fast solvers of nonlinear eigenproblems and higher (than quadratic) order optimization problems. Its application is thus well beyond the Gross-Pitaevskii eigenvalue problem. The operator class and the form of the objective function can be generalized. For example, consider 
\begin{equation}
    -\Delta v + V v + \beta |v|^{2\alpha}v = \lambda v
\end{equation}
for general $\alpha>0$. This ground state equation and the corresponding time-dependent nonlinear Schr\"odinger equation are locally well-posed in $H^1(\R^d)$ as long as $2 \alpha +2 < \frac{2d}{\max \{d-2, 0^+\}}$, see e.g. \cite{frank2014ground} and references therein.
The previous Gross-Pitaevskii eigenvalue problem corresponds to the case $\alpha=1$. 

In general, Theorem \ref{thm:conv} holds true for any $\alpha>0$. The adaptive \zzy{inner product} remains well-posed and the ground state remains a ``double'' eigenstate. The change of \zzy{inner product} from $a_v(\cdot)$ to $a_u(\cdot)$ in the proof of Theorem \ref{thm:GP-L} essentially relies on the convexity of the last term $\int |\cdot|^{2\alpha+2}$ in the energy functional $E(\cdot)$. Therefore, extensions of the previous results in both spatially continuous and discretized cases are easy. We do not present the details here.

It is also common in physics that the \rev{diffusion} is not homogeneous in all spatial directions. For example, it can be stronger in two physical directions and weaker in the third one. 
More generally, we have
\begin{equation}
    -\nabla\cdot(A(x) \nabla v) + V v + \beta |v|^{2\alpha}v = \lambda v
\end{equation}
where the coefficient $A(x) \in L^\infty(\Omega)^{d\times d}$, $A(x)$ is symmetric and coercive. An interesting discrete counterpart to this is the nonlinear Schr\"odinger equation on metric trees (e.g. \cite{dovetta2019nls}), where the Laplacian is replaced by a graph Laplacian on a tree-graph $\mathcal{G}$. 

When restricted to a bounded domain, so that the lowest part of the spectrum is always point spectrum, our previous arguments still hold. In the elliptic case, the discretized $\A_h$ may or may not be an M-matrix, but one can always turn to the small grid size limit $h\to 0^+$ limit when necessary.

For an even broader class of nonlinear eigenproblems or constrained optimization problems, the Sobolev gradient descent may still be applicable, but it is not clear whether \emph{exponential convergence} is still true. It can be seen from previous sections that the convergence rate relies on the (\ref{eq:L}) condition, which in turn relies on the ground state $v$ being the ground state of the linearized operator $\A_v$ at $v$, i.e., the so-called \emph{``double ground state''} property. This is a nontrivial property in general, although it can be true for some operators like the biharmonic operator under certain conditions. 
 
\vspace{6pt}

We discuss here one specific generalization of nonlinear Schr\"odinger eigenproblem, and demonstrate that the Sobolev PGD indeed has the potential of tackling previously formidable problems. The problem of interest is 
\begin{equation}
\label{eq:HOI}
    -\Delta v + V v + \beta |v|^{2}v - \delta \Delta (|v|^2) v  = \lambda v,
\end{equation}
where $\delta \ge 0$. In other words, we add a higher-order interaction term $- \delta \Delta (|v|^2)$ to the Gross-Pitaevskii problem. The corresponding energy functional is 
\begin{align}
\label{eq:HOI_E}
    E(u) = \int |\nabla u|^2+V|u|^2+\frac{\beta}{2} |u|^4 + \frac{\delta}{2}\left|\nabla |u|^2\right|^2.
\end{align}
The above eigenproblem and its variational form are analyzed in \cite{bao2019ground}. Moreover, in \cite{bao2019computing} the authors propose to minimize the energy functional (\ref{eq:HOI_E}) by reformulating it as $E(\rho)=\int |\nabla \sqrt{\rho}|^2+V\rho+\frac{\beta}{2} \rho^2 + \frac{\delta}{2}\left|\nabla \rho\right|^2$, where $\rho:=|u|^2$. This reformulation facilitates the minimization, but it also suffers from the lack of continuity of $|\nabla \sqrt{\rho}|$ near $\rho \to 0^+$. This has to be treated with extra care, and a regularization term has to be added, which complicates the analysis. Therefore, instead of replacing $|u|^2$ with $\rho$, we propose to minimize $E(u)$ with respect to $u$ directly with the Sobolev PGD. 

Assume that Assumptions \ref{ass:GP} still hold. Define the manifold $\M$ with an extra constraint:
\begin{align*}
    \M := \left\{ z\in H_0^1(\Omega): \, \|u\|_{L^2}=1, \, \|u\|_{L^\infty}\le M_0, \, \|\nabla u\|_{L^\infty}\le M_1 \right\}.
\end{align*}
Define the adaptive linearized operator and the respective inner products as follows:
\begin{align*}
    &(z, w)_{a_u} := \int_{\Omega} \nabla z \nabla w + V z w + \beta u^2 z w + \delta \nabla(u z) \nabla(u w),  \\
    &(z, \A_u w)_{L^2} := (z, w)_{a_u}, \\
    &(z, w)_{a_0} := \int_{\Omega} \nabla z \nabla w + V z w,  \qquad \forall\, z,\,w \in \T_u\M,\quad u\in\M.
\end{align*}
Then we have the following results. 
\begin{lemma}
    The ground state $v$ of (\ref{eq:HOI}) satisfies $v>0$ everywhere on $\Omega$. It is the unique positive eigenstate of (\ref{eq:HOI}). It is also both the unique ground state of  (\ref{eq:HOI}) and that of the linearized operator $\A_v$ up to the sign.
\end{lemma}
\begin{proof}
    Following the same arguments as in Lemma 2 in \cite{cances2010numerical}, the extended $E(u)$ as in (\ref{eq:HOI_E}) still admits a nonnegative minimizer $v$. According to \cite[Theorem~2.2]{bao2019ground}, we know that $v \in C^{1,1}(\bar{\Omega})$. This implies that $v$, $\nabla v \in L^\infty(\Omega)$. Thus, the nonnegative $v$ can still be made positive by the Harnack inequality. Also, the linearized operator $\A_v$ still has a unique positive ground state, which is exactly the above $v$. Thus the ``double ground state'' property remains true.
    
    We now show that (\ref{eq:HOI}) has a unique positive eigenstate by a contradiction argument. Suppose instead that there is a different positive eigenstate $\tilde{v}>0$ with its eigenvalue $\tilde{\lambda}$, and $E(\tilde{v}) > E(v)$. Using the Picone identity, $\int \nabla \tilde{v} \nabla (\frac{v^2}{\tilde{v}}) \le \int (\nabla v)^2$. We have
    \begin{align*}
        & \tilde{\lambda} - \lambda = \tilde{\lambda} (v,v)_{L^2} - (v,v)_{a_v}
        = \tilde{\lambda} \left(\tilde{v}, \frac{v^2}{\tilde{v}}\right)_{L^2} - (v,v)_{a_v}
        =  \left(\tilde{v},\frac{v^2}{\tilde{v}}\right)_{a_{\tilde{v}}} - (v,v)_{a_v} \\
        & = \int \nabla \tilde{v} \cdot \nabla \left(\frac{v^2}{\tilde{v}}\right) + V v^2 + \beta  \tilde{v}^2v^2 + \delta \nabla (\tilde{v}^2) \nabla (v^2) - \int (\nabla v)^2+Vv^2 + \beta v^4 + \delta (\nabla (v^2))^2 \\
        & \le \int (\nabla v)^2 + V v^2 + \frac{\beta}{2}(v^4 + \tilde{v}^4) +\frac{\delta}{2}\left((\nabla(v^2))^2+(\nabla (\tilde{v}^2))^2\right) - \int (\nabla v)^2+Vv^2 + \beta v^4 + \delta (\nabla (v^2))^2 \\
        & = \int  \frac{\beta}{2} \tilde{v}^4 +\frac{\delta}{2}(\nabla (\tilde{v}^2))^2 - \int \frac{\beta}{2} v^4 + \frac{\delta}{2} (\nabla (v^2))^2 = (\tilde{\lambda} - E(\tilde{v})) - (\lambda - E(v)),
    \end{align*}
    i.e.,
    \begin{align*}
        E(\tilde{v}) \le E(v).
    \end{align*}
    This contradicts our assumption that $E(\tilde{v}) > E(v)$.
\end{proof}

The next lemma shows that the eigenvalue and eigenfunction perturbation results stated in Lemma \ref{lemma:perturb} hold similarly for (\ref{eq:HOI}).
\begin{lemma}
\label{lemma:HOI}
    Let $\lambda_i$ and $\mu_i$ be the $i$-th smallest eigenvalues of $\A_v$ and $\A_u$ respectively, and $v_i$ and $w_i$ be their corresponding eigenvectors (so that $v = v_1$). Let $C_v := \lambda_2 - \lambda_1$ denote the eigenvalue gap. Then there exists a positive constant $C  = C(\beta, \delta, V, M_0, M_1, \Omega, \lambda_1, C_v)$, such that for all $\|u-v\|_{H^1}<C$, $u\in\M$, we have $\|u-w_1\|_{L^2}\le s$ for some $s<1$. 
\end{lemma}
\begin{proof}
    See Appendix \ref{sec:proof_lemma:HOI}.
\end{proof}

\begin{theorem}
    If the initial state satisfies $u_0 \ge 0$ everywhere on $\Omega$, then $\{u_n\}_{n=0}^\infty$ generated by the Sobolev PGD with step size $0<\tau_{\text{min}}\le \tau_n \le \tau_{\text{max}}$ converges to the unique ground state $v$ of (\ref{eq:HOI}) with an asymptotic exponential convergence rate.
\end{theorem}
\begin{proof}
    First, the Sobolev PGD sequence starting from a positive initial value remains positive as before, and convexity ensures convergence to a nonnegative local minimizer of $E(u)$, which must also be the global minimizer and the ground state of (\ref{eq:HOI}). This convergence can be proved to be a strong $H^1$ convergence by the Sobolev embedding and the convergence of energy. 
    
    In order to establish exponential convergence, it suffices to show that Conditions (\ref{eq:L}), (\ref{eq:D}) and (\ref{eq:S}) all hold for $\{u_n\}_{n=0}^\infty$. {The nonnegativity of $\delta$ ensures the equivalence of $a_0$ and $a_u$ norms.} Thus Conditions (\ref{eq:D}) and (\ref{eq:S}) hold. Condition (\ref{eq:L}) follows from Lemma \ref{lemma:HOI} and Lemma \ref{lemma:linear}.
\end{proof}

The above results establish the exponential convergence of the Sobolev PGD for problem (\ref{eq:HOI}) for any $\delta \ge 0$. 
Numerical evidence shows that the Sobolev PGD for this problem converges very well just as the original Gross-Pitaevskii eigenproblem. This is a demonstration that the Sobolev gradient descent has the potential to be generalized to study some continuous or discrete high degree optimization problems. We believe that this method has the potential to be extended to a broader class of problems as long as certain assumptions are satisfied, which is left for our future work.

\section{Numerical experiments}
\label{sec:numerical}


In this section, we demonstrate the convergence of the Sobolev PGD method using some numerical examples.
We show that exponential convergence rate is attained both for the original eigenproblem (\ref{eq:eig_original}) and for its extension (\ref{eq:HOI}). We also observe and discuss some interesting phenomena that one may encounter in numerical experiments.

\subsection{Gross-Pitaevskii eigenproblem in 2D.}
We first look at the Gross-Pitaevskii eigenproblem (\ref{eq:eig_original}) in two dimensions. Let the domain be $\Omega = [-1, 1]^2 \subset \R^2$ with Dirichlet boundary condition. \rev{The problems are discretized with P1 Lagrange finite element method. The grid is a uniform grid with fixed size $h = 2\cdot 2^{-8}$ throughout this section.}

The first example is a single well potential $V(x) = \frac{1}{2}|x|^2$.
It is well known that the Anderson localization \cite{anderson1958absence} is present in this setting. We set $\beta = 1$. The initial guess $z_0$ is chosen as the eigenvector corresponding to the smallest eigenvalue of $\A_0$. It is strictly positive over the whole domain $\Omega$. The step size is $\tau = 1$. 

Figure \ref{Well_N_256_beta_1-1} shows the profile of the potential $V(x)$. Figure \ref{Well_N_256_beta_1-2} is the profile of the computed ground state with $\beta = 1$. Figure \ref{Well_N_256_beta_1-3} displays the \rev{log $H^1$-error convergence $\text{log}_{10}(\|u_n-v\|_{H^1}/\|v\|_{H^1})$}. It can be seen that the Sobolev PGD converges in just a few steps with an exponential (linear) convergence rate.  

\begin{figure}[ht]
    \centering
    \begin{subfigure}[t]{.32\linewidth}
        \includegraphics[width = \linewidth, align = c]{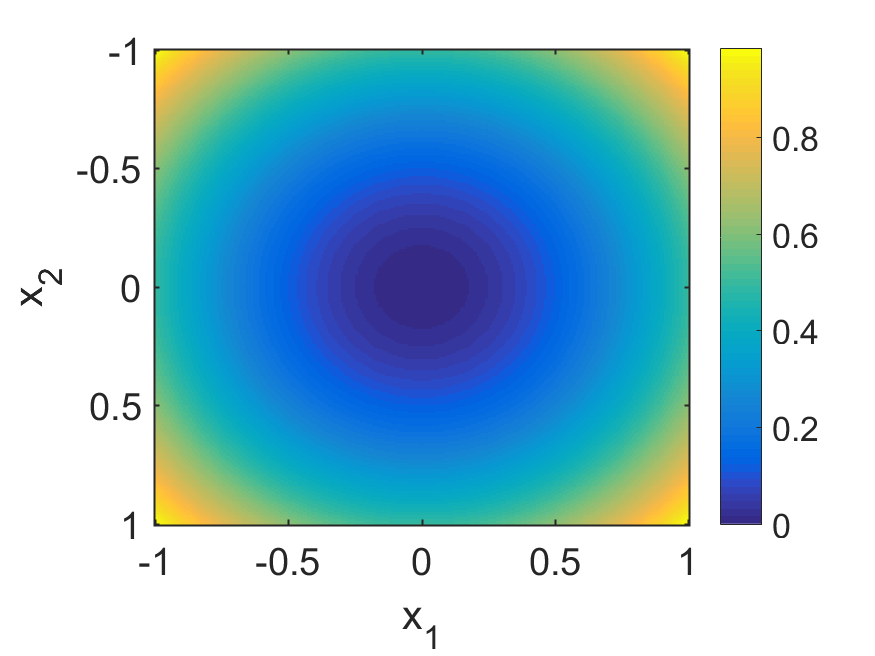}
        \caption{Single well potential $V(x) = \frac{1}{2}|x|^2$}
        \label{Well_N_256_beta_1-1}
    \end{subfigure}
    \hspace{0.01\textwidth}
    \begin{subfigure}[t]{.32\linewidth}
        \includegraphics[width = \linewidth, align = c]{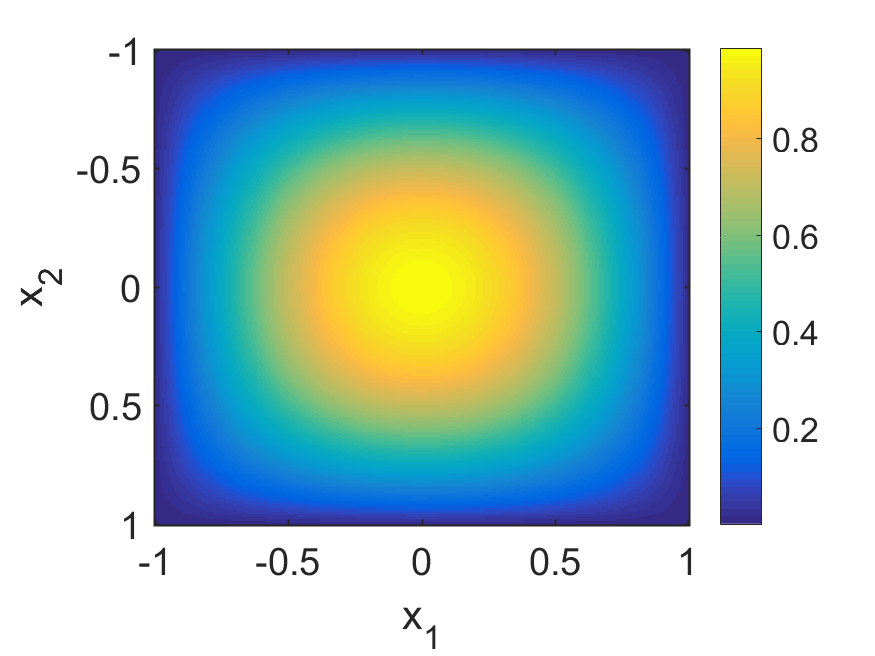}
        \caption{Ground state when $\beta =1$ }
        \label{Well_N_256_beta_1-2}
    \end{subfigure}
    \hspace{0.02\textwidth}
    \begin{subfigure}[t]{.29\linewidth}
        \includegraphics[width = \linewidth, align = c]{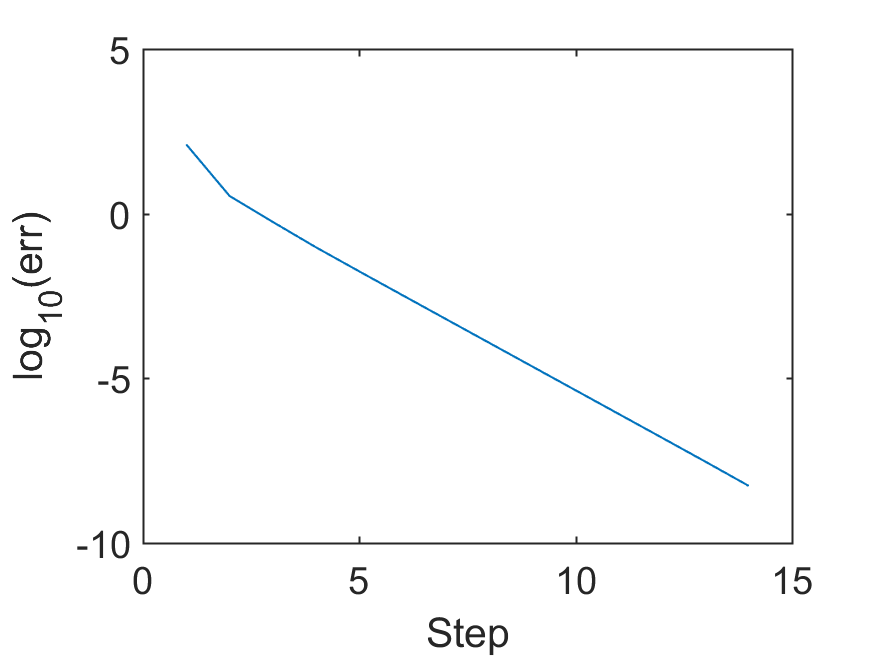}
        \caption{\rev{Log $H^1$-error convergence}}
        \label{Well_N_256_beta_1-3}
    \end{subfigure}
    \caption{Example of (\ref{eq:eig_original}) with single well potential $V = \frac{1}{2}|x|^2$ and $\beta = 1$.}
\end{figure}

By increasing $\beta$, there is a greater nonlinearity in the problem. When $\beta \gg 1$, the quartic term $\frac{\beta}{2} |u|^4$ would dominate the energy functional (\ref{eq:E}). This would be a significant barrier to some traditional methods. Yet the Sobolev PGD remains stable and fast. Figures \ref{Well_N_256_beta_10-2} to \ref{Well_N_256_beta_100-3} show the \rev{log $H^1$-error convergence} and the profiles of the respective ground states with $\beta = 10$ and $\beta = 100$ respectively. With the Sobolev PGD, there is only a mild increase in the computational complexity, and the iteration still converges exponentially fast as predicted.

\begin{figure}[ht]
    \centering
    \begin{subfigure}[t]{.32\linewidth}
        \includegraphics[width = \linewidth, align = c]{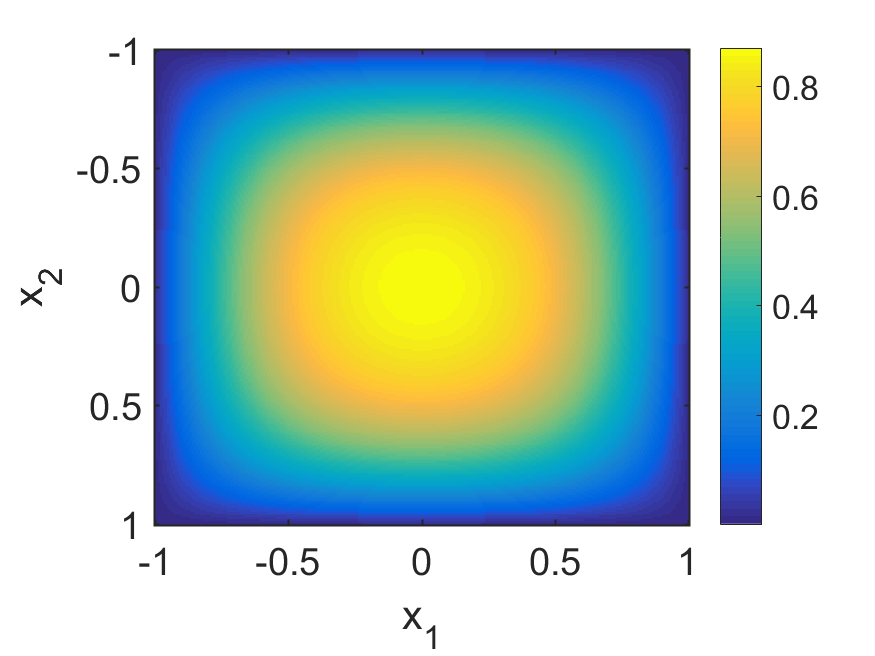}
        \caption{Ground state when $\beta = 10$}
        \label{Well_N_256_beta_10-2}
    \end{subfigure}
    \hspace{0.05\textwidth}
    \begin{subfigure}[t]{.29\linewidth}
        \includegraphics[width = \linewidth, align = c]{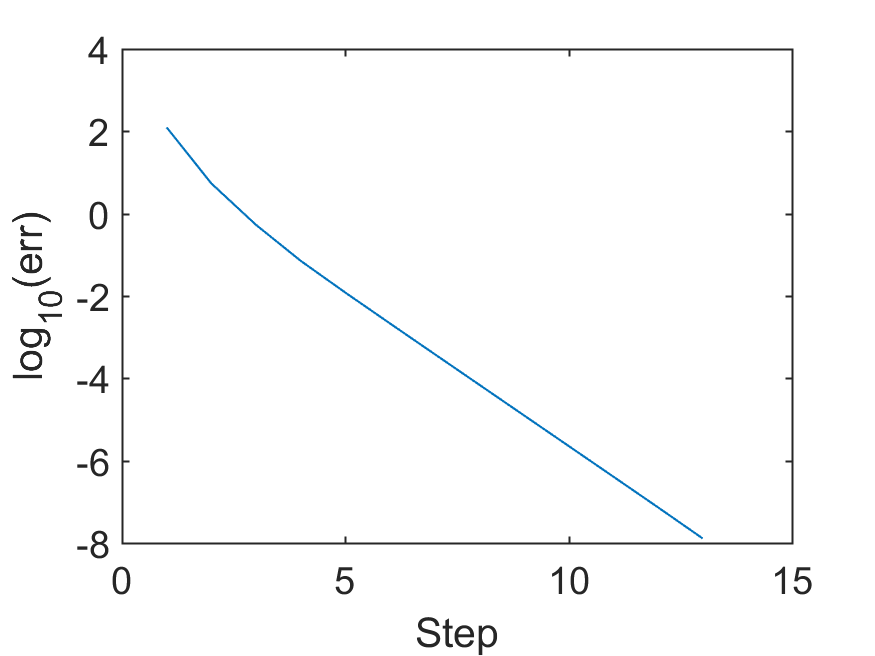}
        \caption{Convergence when $\beta = 10$}
        \label{Well_N_256_beta_10-3}
    \end{subfigure}
    \begin{subfigure}[t]{.32\linewidth}
        \includegraphics[width = \linewidth, align = c]{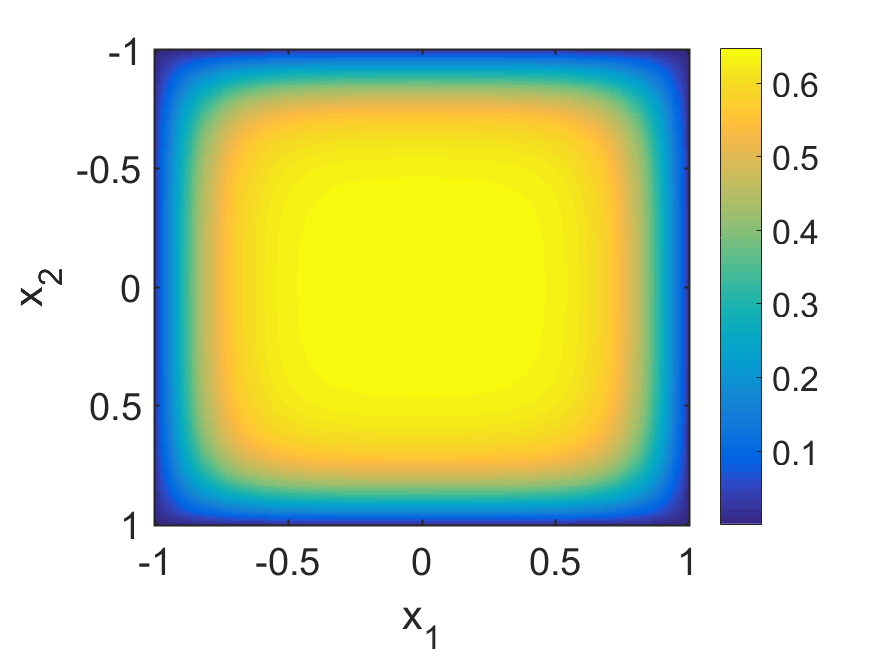}
        \caption{Ground state when $\beta = 100$}
        \label{Well_N_256_beta_100-2}
    \end{subfigure}
    \hspace{0.05\textwidth}
    \begin{subfigure}[t]{.29\linewidth}
        \includegraphics[width = \linewidth, align = c]{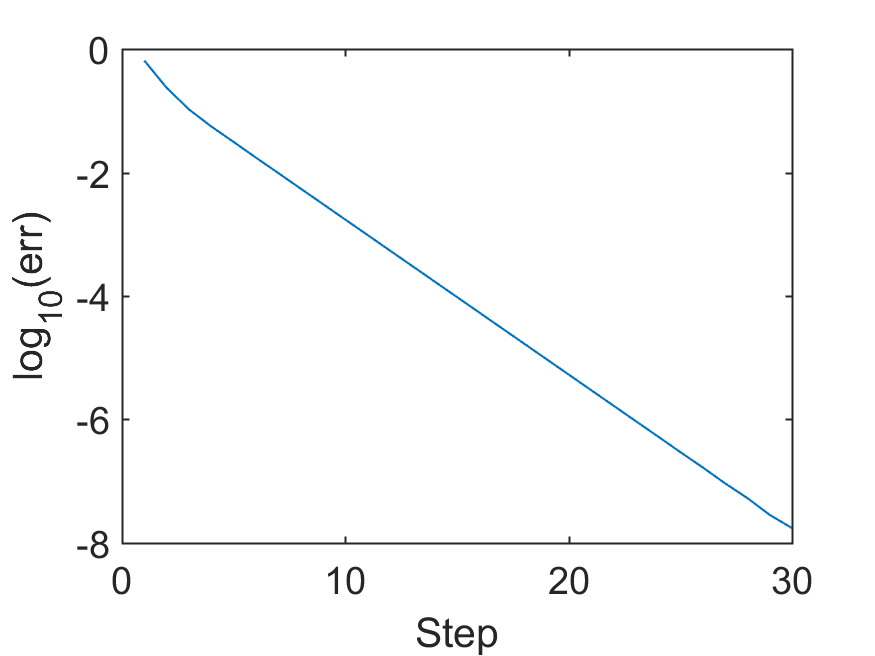}
        \caption{Convergence when $\beta = 100$}
        \label{Well_N_256_beta_100-3}
    \end{subfigure}
    \caption{Example of (\ref{eq:eig_original}) with single well potential $V = \frac{1}{2}|x|^2$ and $\beta = 10$ or $100$.}
\end{figure}


\subsection{Localization under the disordered potential.}
The second example is a disordered potential $V$. Its fully discrete counterpart, the randomized potential on the lattice $\mathbb{Z}^d$, has been extensively studied for its rich behaviour in spectral gaps, exponential localization of eigenstates near the bottom of the spectrum, and implications about the ``mobility edge'' conjecture in quantum physics and random matrix theory \cite{frohlich1983absence,deift2017some}. 

In our semi-lattice example, the localization of the ground state is also present. In the experiment, $V(x)$ is generated as follows. The extent of disorder is determined by a parameter $K = 50$. This means that the domain $\Omega$ is divided into $K\times K$ cells. The value of V(x) in each cell is either $1$ or $1/K^2$, randomly chosen with equal probability.

Figure \ref{Disorder_N_256_beta_0.5-1} shows the profile of $V(x)$. Figure \ref{Disorder_N_256_beta_0.5-5} displays the computed ground state with $\beta = 0.5$. It can be seen that the ground state is concentrated in a small region whose diameter is about a few times the interaction length of the disorder. Figure \ref{Disorder_N_256_beta_0.5-3} shows the convergence rate of the Sobolev PGD iteration for this example. 

To facilitate convergence, we have chosen $\tau = 1.5$. Although Corollary \ref{col:global} requires a small $\tau$, in the numerical experiments we find that choosing $\tau >1$ results in significantly faster convergence. This is in accordance with the empirical findings of \cite{henning2020sobolev}.

\begin{figure}[ht]
    \centering
    \begin{subfigure}[t]{.32\linewidth}
        \includegraphics[width = \linewidth, align = c]{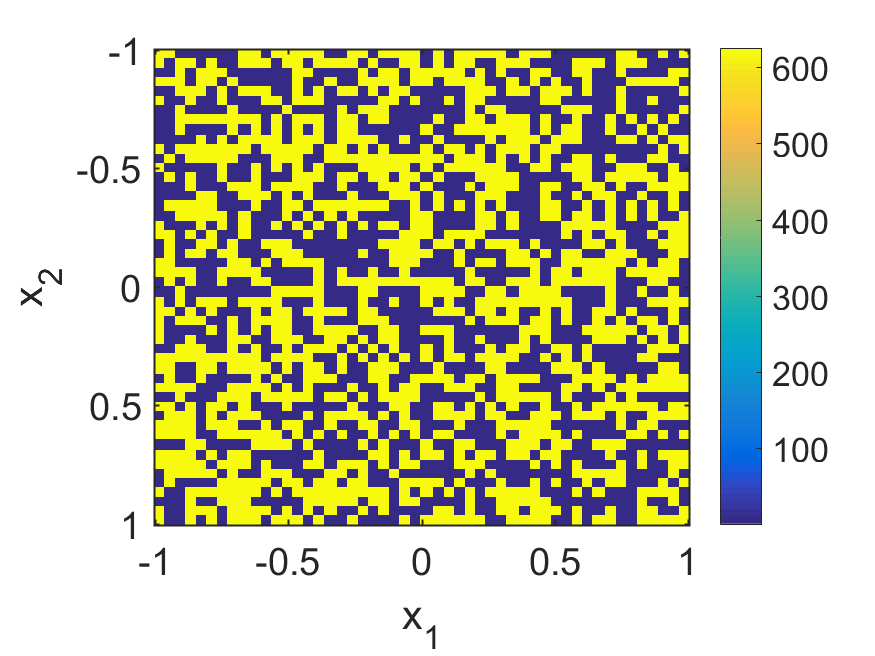}
        \caption{Profile of the disordered potential}
        \label{Disorder_N_256_beta_0.5-1}
    \end{subfigure}
    \hspace{0.01\textwidth}
    \begin{subfigure}[t]{.32\linewidth}
        \includegraphics[width = \linewidth, align = c]{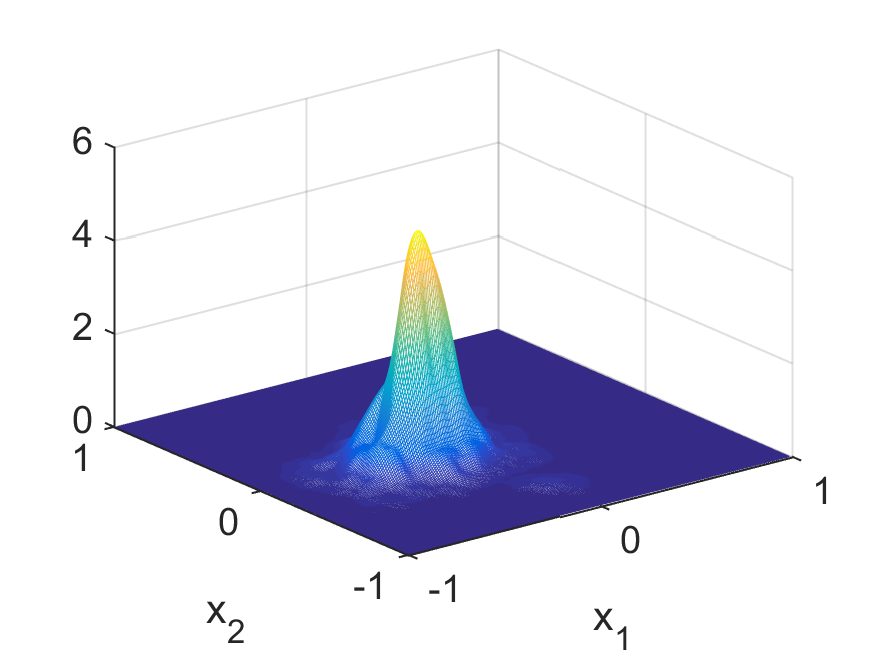}
        \caption{Profile of the ground state}
        \label{Disorder_N_256_beta_0.5-5}
    \end{subfigure}
    \hspace{0.01\textwidth}
    \begin{subfigure}[t]{.29\linewidth}
        \includegraphics[width = \linewidth, align = c]{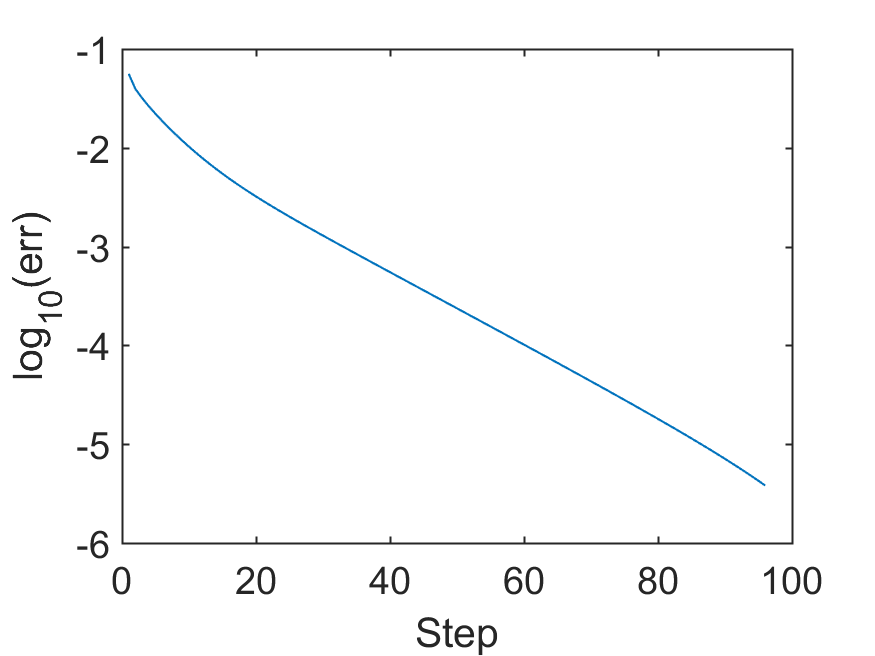}
        \caption{\rev{Log $H^1$-error convergence}}
        \label{Disorder_N_256_beta_0.5-3}
    \end{subfigure}
    \caption{Example of (\ref{eq:eig_original}) with a disordered potential and $\beta = 0.5$.}
\end{figure}


\subsection{Asymptotic escape of Sobolev PGD from saddle states.}
It is interesting to look at the asymptotic behaviour of the Sobolev gradient descent method if starting from a non-positive initial value. Recall that Corollary \ref{col:global} only ensures exponential convergence to the global ground state from $u_0\ge 0$. When this condition is violated, it is a priori unknown what the iteration will converge to. It is possible that there are other spurious fixed points, including local minimizers and saddle points. The first-order condition ensures that all these spurious fixed points are eigenstates. But the convergence rate to such points is unknown. 

As for the spatially discretized case, the Hilbert manifold $\M$ becomes a Riemannian manifold, and the spectra of the operators become finite. As is proved in \cite{paper1} and references therein, a random initialization almost surely avoids saddles and converges only to local or global minimizers. It means that if an excited state is a strict saddle point, then a random initialization is very unlikely to converge to that state. As for the spatially continuous case, it is reasonable to expect the same phenomenon, although the analysis could be more difficult due to the infinite dimension of $\M$ and the infinite number of eigenstates.

In the subsequent numerical experiments, we let $V(x) = \frac{1}{2}|x|^2$ and $\beta = 100$. We will use an example to show that for an excited state that is a strict saddle, it has a very thin converging set close to measure zero. Thus, using Sobolev PGD to compute excited states could be unstable. The accuracy of the computed excited states could be limited.

First, we let the initialization $u_0$ be the second-smallest eigenvector of $\A_0$. This $u_0$ is positive on half of $\Omega$ and negative on the other half. It is displayed in Figure \ref{Np_N_256_beta_100-4}. We then let Sobolev PGD iterate a few steps. Figure \ref{Np_N_256_beta_100-5} displays the computed state $u^*$ when the algorithm stops. Figure \ref{Np_N_256_beta_100-3} shows the decrease of the log $L^2$ error with respect to $u^*$. We also compute the manifold Hessian at $u^*$ and find that it has at least one negative eigenvalue. Thus $u^*$ is a strict saddle state.

\begin{figure}[ht]
    \centering
    \begin{subfigure}[t]{.32\linewidth}
        \includegraphics[width = \linewidth, align = c]{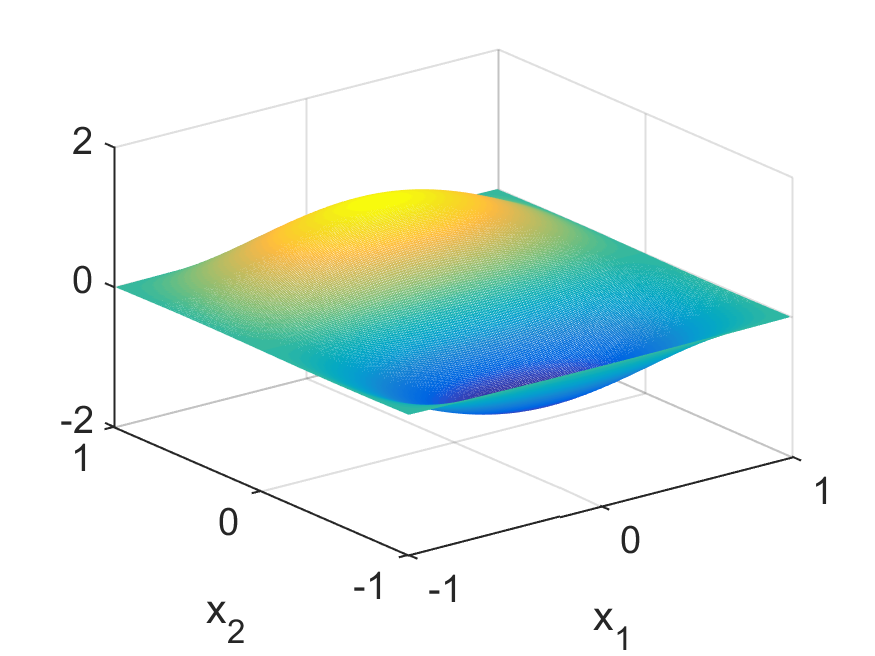}
        \caption{Profile of initial state $u_0$}
        \label{Np_N_256_beta_100-4}
    \end{subfigure}
    \hspace{0.01\textwidth}
    \begin{subfigure}[t]{.32\linewidth}
        \includegraphics[width = \linewidth, align = c]{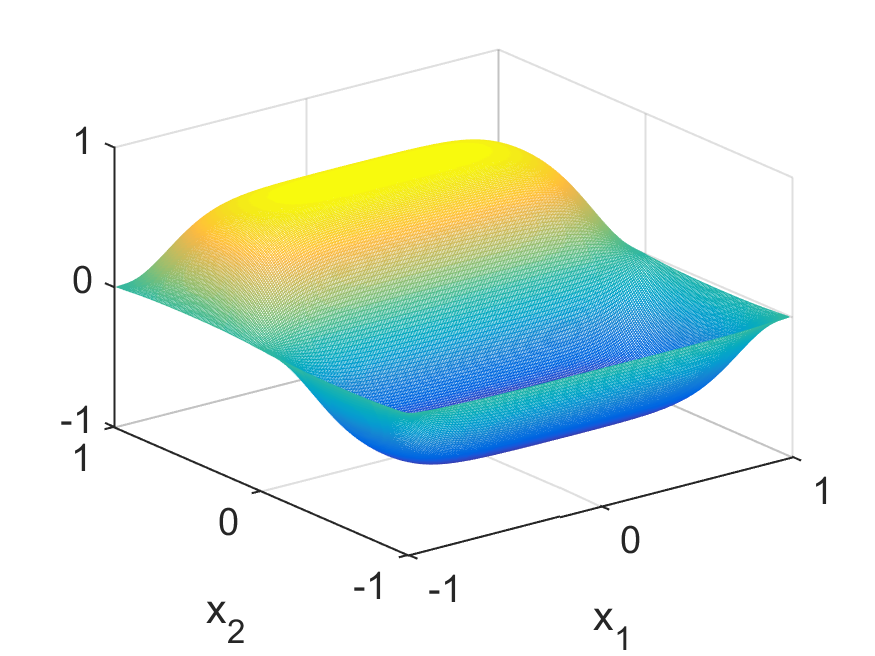}
        \caption{Profile of computed state $u^*$}
        \label{Np_N_256_beta_100-5}
    \end{subfigure}
    \hspace{0.01\textwidth}
    \begin{subfigure}[t]{.29\linewidth}
        \includegraphics[width = \linewidth, align = c]{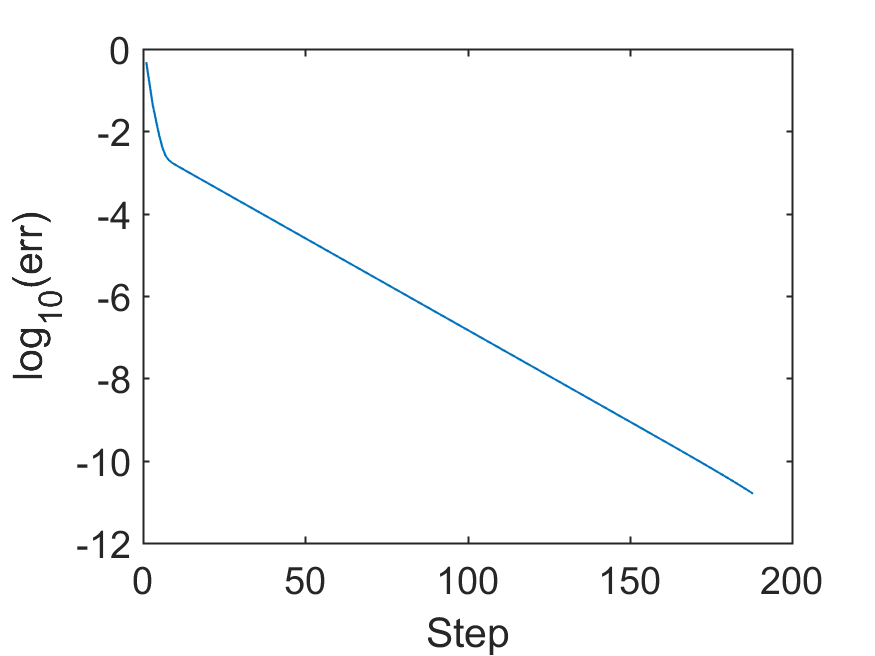}
        \caption{\rev{Log $H^1$-error convergence}}
        \label{Np_N_256_beta_100-3}
    \end{subfigure}
    \caption{Behavior of Sobolev PGD for (\ref{eq:eig_original}) with $u_0 \not\ge 0$.}
\end{figure}   

Next, we add a small perturbation to $u_0$: we let $\hat{u}_0 = u_0 + \epsilon\cdot \eta$, where $\eta$ is a random noise that is of the same order as $u_0$, and the parameter $\epsilon$ controls the magnitude of noise. We let Sobolev PGD start from $\hat{u}_0$ and trace its evolution. What we observe is that, as long as there is a small perturbation, Sobolev PGD escapes from the previous saddle state and converges to the ground state. The parameter $\epsilon$ can be chosen as small as $10^{-4}$ and this effect is still present. 

Specifically, Figures \ref{Np_N_256_beta_100_perturb_1E-2_saddle-2} to \ref{Np_N_256_beta_100_perturb_1E-4_saddle-2} demonstrate the evolution of the log-distance to the precomputed closest excited state $u^*$. We choose $\epsilon = 10^{-2}$, $10^{-3}$, and $10^{-4}$, respectively. Saddle escape behavior can be observed in all three cases. We can see that the distance to the excited state first goes down, then goes up. Figure \ref{Np_N_256_beta_100_perturb_1E-4-5} show the computed state starting from $\hat{u}_0$, and it is the ground state.

In general, first-order optimization methods, including Sobolev PGD as well as other methods in the gradient descent family, are not good choices for the computation of excited states. They rely on a good enough initialization (like the above $u_0$ without noise) and could suffer from numerical instability issues. One has to resort to other methods if the goal is to obtain high accuracy. We will explore this topic in our upcoming work.

\begin{figure}
    \begin{subfigure}[t]{.31\linewidth}
        \includegraphics[width = \linewidth, align = c]{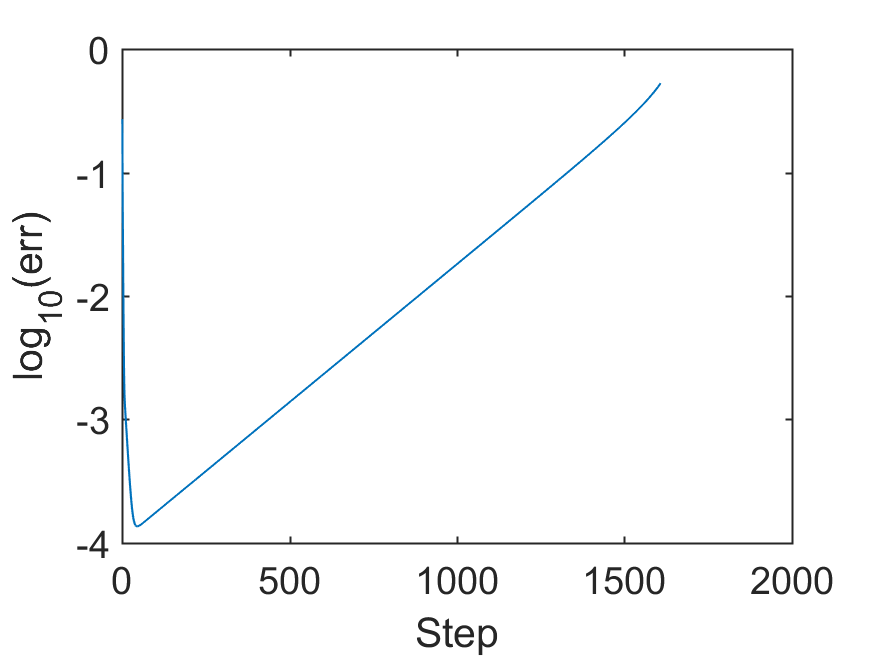}
        \caption{Noise level $\epsilon = 10^{-2}$}
        \label{Np_N_256_beta_100_perturb_1E-2_saddle-2}
    \end{subfigure}
    \hspace{0.01\textwidth}
    \begin{subfigure}[t]{.31\linewidth}
        \includegraphics[width = \linewidth, align = c]{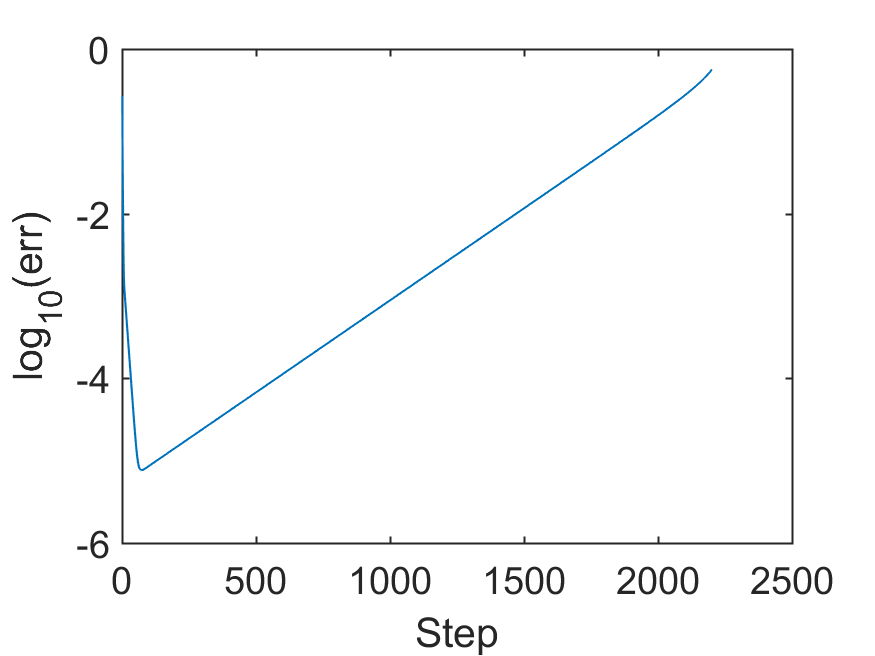}
        \caption{Noise level $\epsilon = 10^{-3}$}
        \label{Np_N_256_beta_100_perturb_1E-3_saddle-2}
    \end{subfigure}
    \hspace{0.01\textwidth}
    \begin{subfigure}[t]{.31\linewidth}
        \includegraphics[width = \linewidth, align = c]{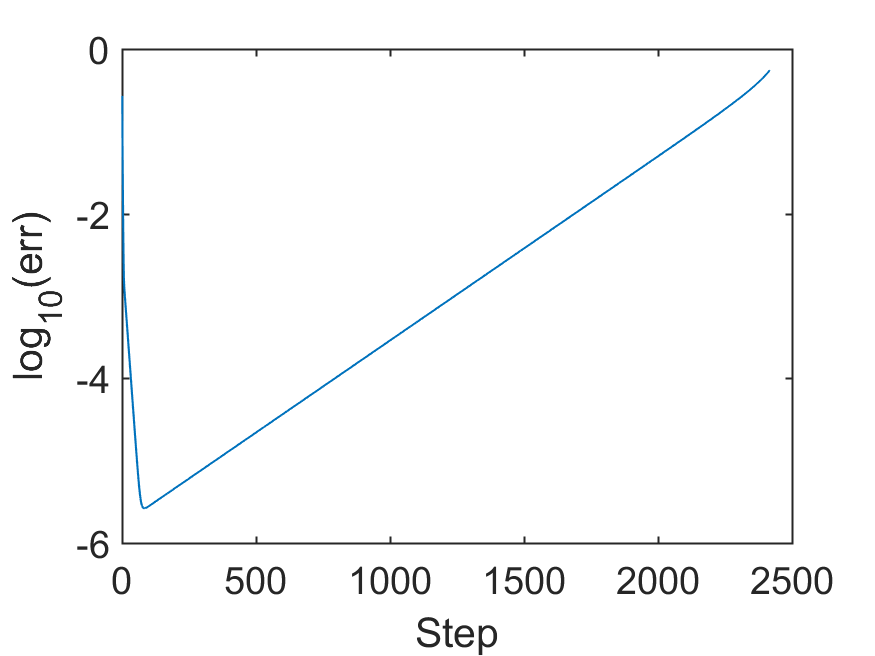}
        \caption{Noise level $\epsilon = 10^{-4}$}
        \label{Np_N_256_beta_100_perturb_1E-4_saddle-2}
    \end{subfigure}
    
    \hspace{0.16\textwidth}
    \begin{subfigure}[t]{.32\linewidth}
        \includegraphics[width = \linewidth, align = c]{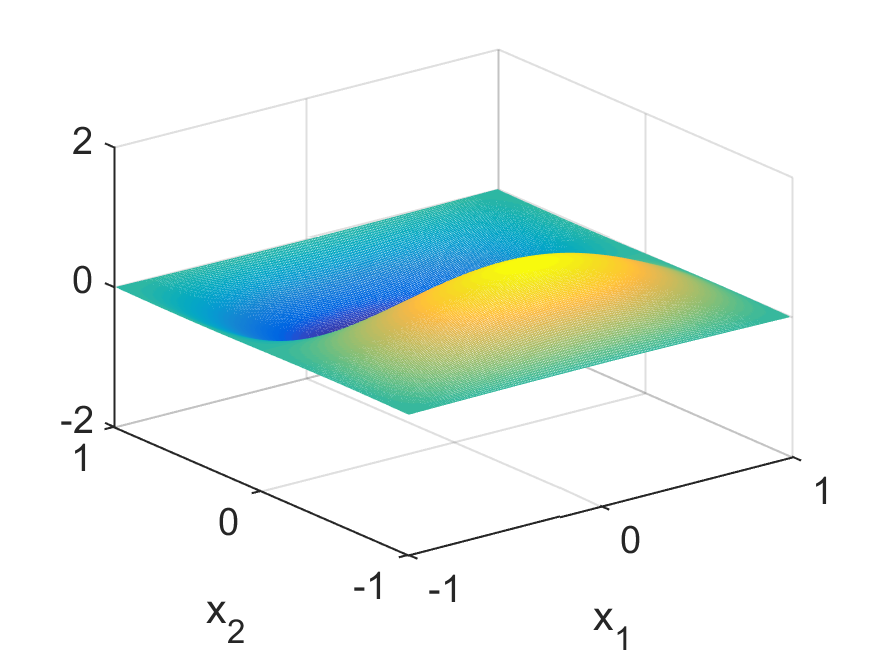}
        \caption{Profile of $\hat{u}_0 = u_0 + 10^{-4}\eta$}
        \label{Np_N_256_beta_100_perturb_1E-4-4}
    \end{subfigure}
    \hspace{0.01\textwidth}
    \begin{subfigure}[t]{.32\linewidth}
        \includegraphics[width = \linewidth, align = c]{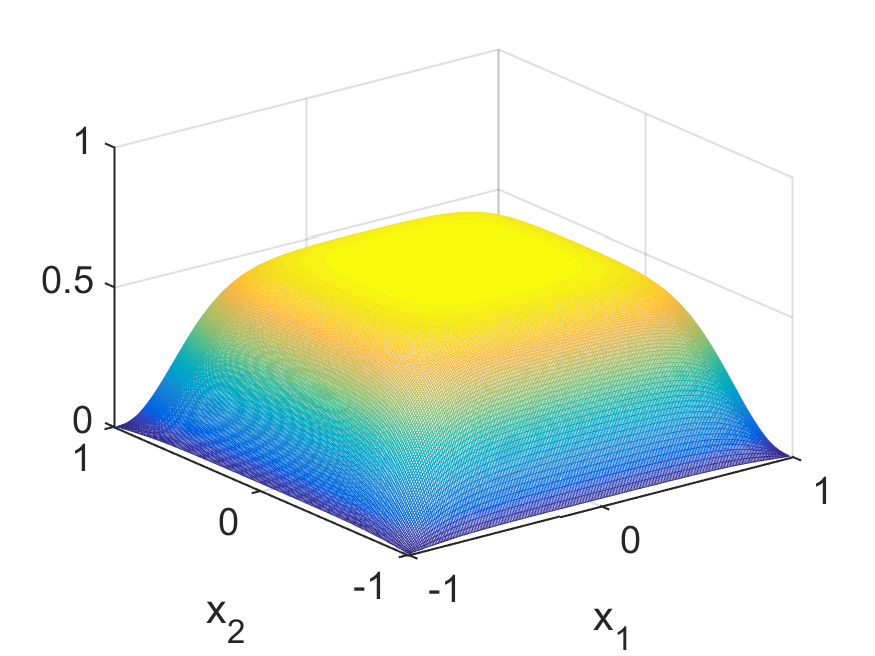}
        \caption{Profile of computed state starting from $\hat{u}_0$}
        \label{Np_N_256_beta_100_perturb_1E-4-5}
    \end{subfigure}
    \caption{Asymptotic escape from saddle state under small perturbations. Figures (a)-(c) displays the distances to the saddle state $u^*$ starting from $\hat{u}_0 = u_0 + \epsilon\cdot \eta$.}
\end{figure}

\subsection{High order interaction.}
We now look at Problem (\ref{eq:HOI}) with an extra high order interaction term. This adds additional nonlinearity to the problem. Consider the same domain $\Omega = [-1,1]^2 \subset \R^2$ and spatial discretization size $h = 2 \cdot 2^{-8}$. Let $V(x) = \frac{1}{2}|x|^2$ still be the single well potential. The first example is $\beta = 10$ and $\delta = 1$. 
Figure \ref{HOI_N_256_beta_10_delta_1-3} shows the log error convergence. The iteration converges in a few steps and shows a good convergence rate.

In the second example, we increase $\delta$ and look at the problem with strong high order interaction. We choose $\beta = 100$ and $\delta = 100$. Figure \ref{HOI_N_256_beta_100_delta_100-3} shows the log error convergence. The convergence rate is slower but stable.

\begin{figure}[ht]
    \centering
    \begin{subfigure}[t]{.32\linewidth}
        \includegraphics[width = \linewidth, align = c]{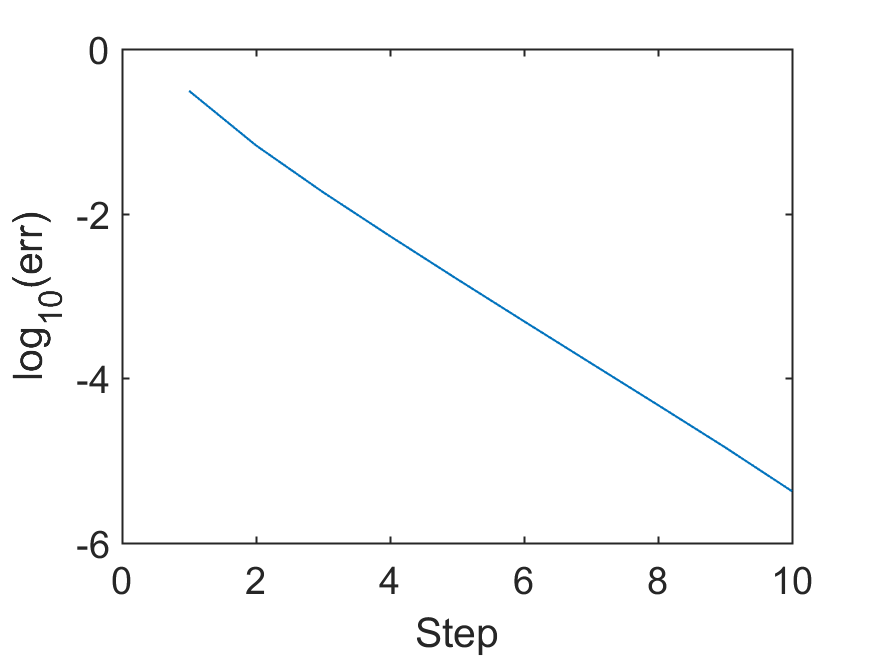}
        \caption{Convergence when $\beta = 10$, $\delta = 1$}
        \label{HOI_N_256_beta_10_delta_1-3}
    \end{subfigure}
    \hspace{0.05\textwidth}
    \begin{subfigure}[t]{.32\linewidth}
        \includegraphics[width = \linewidth, align = c]{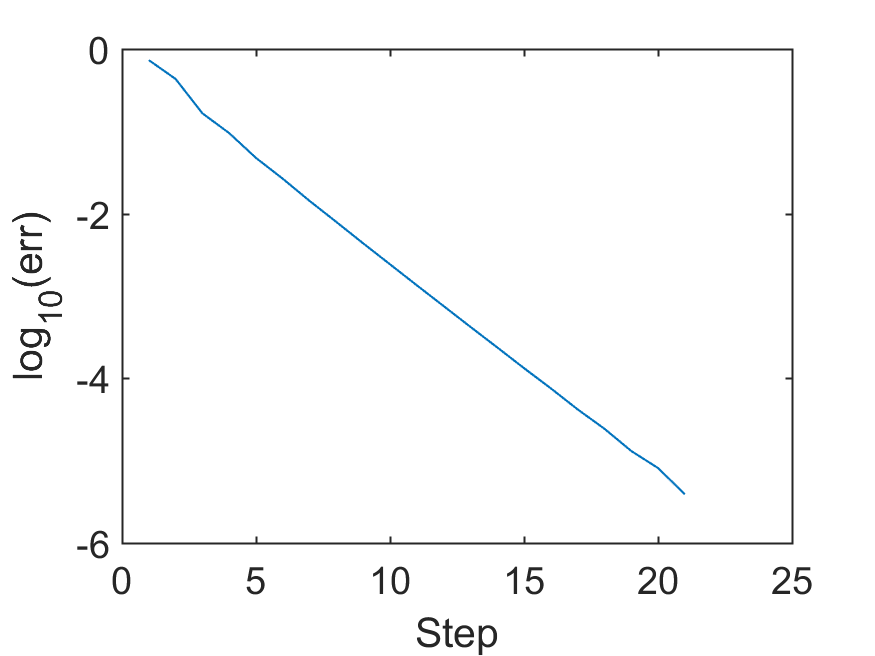}
        \caption{Convergence when $\beta = 100$, $\delta = 100$}
        \label{HOI_N_256_beta_100_delta_100-3}
    \end{subfigure}
    \caption{Examples of (\ref{eq:HOI}) with different nonlinear effects}
\end{figure}

\section{Conclusion}
\label{sec:conclusion}

In this paper, we analyzed the exponential convergence of the $a_u$-Sobolev gradient descent method without resorting to the time-continuous gradient flow. To this purpose, we introduced a general convergence tool using the {\L}ojasiewicz inequality, and adapted it to the setting of infinite dimensional Hilbert manifold and mixed norms. By proving the (\ref{eq:L}), (\ref{eq:D}) and (\ref{eq:S}) conditions for the Sobolev PGD, we were able to unveil the mechanism behind the good performance of the Sobolev PGD for the Gross-Pitaevskii eigenproblem (\ref{eq:eig_original}), which was only empirically observed in previous works. 

The success of the Sobolev PGD on the Gross-Pitaevskii eigenproblem inspires us to further explore alternative fast solvers for more general nonlinear eigenproblems and optimizations with high degree objective functions. Our analysis revealed that the essential condition is the ``double ground state'' property, namely the ground state of the nonlinear problem is also the unique ground state of the linearized operator at that point. This can be rigorously proved in some cases and seems to be true in a number of physical applications of interest based on empirical evidence. Specifically, we showed that this condition is satisfied for a nonlinear Schr\"odinger eigenproblem with extra high order interaction term. Thus the Sobolev PGD works well for this problem and has superiority over previous methods. 

\vspace{8pt}

{\textbf{Acknowledgements.}} This research was in part supported by NSF grants DMS-1912654 and DMS-1907977. The author would like to thank Thomas Y. Hou for the helpful comments on earlier versions of this work, and Zhenzhen Li for introducing the \L ojasiewicz inequality to the author. The author would also like to acknowledge the warm hospitality of Oberwolfach Research Institute for Mathematics during the seminar \emph{Beyond Numerical Homogenization}, where the early ideas of this work started.

\bibliographystyle{plain}
\bibliography{reference}

\appendix

\section{Proofs of technical lemmas}
\subsection{Proof of Lemma \ref{lemma:equiv}}
\label{sec:proof_lemma:equiv}

\begin{proof}
     As for the equivalence between $\|\cdot\|_{a_0}$ and $\|\cdot\|_{a_u}$, the second part of the inequality holds for all $0 < C_E \le 1$ since $u^2$ is nonnegative. For the first part, by Poincar\'e inequality, $\|z\|_{L^2}^2 \le C_P |z|_{H^1}^2$ for some domain constant $C_P = C_P(\Omega)$. Thus, we have
    \begin{align*}
        \|z\|_{a_0}^2 - C_E \|z\|_{a_u}^2 &= (1-C_E)|z|_{H^1}^2 + \int_{\Omega} ((1-C_E)V-C_E\beta u^2) z^2 \\
        &\ge (1-C_E) |z|_{H^1}^2 - C_E\beta   \int_{\Omega} u^2 z^2\\
        &\ge (1-C_E - C_E\beta M_0^2 C_P)|z|_{H^1}^2,   \qquad \forall z \in H_0^1(\Omega),\quad C_E \le 1.
    \end{align*}
    Take $0< C_E \le 1/(1 + \beta M_0^2 C_P)$, then $C_E \|z\|_{a_u}^2 \le \|z\|_{a_0}^2$. 
    
    As for the equivalence between $\|\cdot\|_{a_u}$ and $\|\cdot\|_{H^1}$, we have
    \begin{align*}
        \|z\|_{H^1}^2 - \widetilde{C_E} \|z\|_{a_u}^2 &= \|z\|_{H^1}^2 - \widetilde{C_E}|z|_{H^1}^2 - \widetilde{C_E}\int_{\Omega}  (V+\beta u ^2) z^2 \\
        &\ge \left(1-\widetilde{C_E}  - \widetilde{C_E}C_P(\|V\|_{L^\infty} + \beta M_0^2) \right) |z|_{H^1}^2, 
        \qquad \forall z \in H_0^1(\Omega),\quad \widetilde{C_E} \le 1.
    \end{align*}
    Take $0<\widetilde{C_E}\le 1/(1+C_P(\|V\|_{L^\infty}+\beta M_0^2))$, then $\widetilde{C_E}\|z\|_{a_u}^2 \le \|z\|_{H^1}^2$. On the other hand,
    \begin{align*}
        \widetilde{C_E}^{-1}\|z\|_{a_u}^2 -  \|z\|_{H^1}^2 &= (\widetilde{C_E}^{-1}-1)|z|_{H^1}^2 + \widetilde{C_E}^{-1}\int_{\Omega}  (V+\beta u ^2) z^2 - \|z\|_{L^2} \\
        & \ge \left(C_P^{-1}(\widetilde{C_E}^{-1}-1) + \widetilde{C_E}^{-1}\beta M_0^2 - 1 \right)\|z\|_{L^2}.
    \end{align*}
    Take $0<\widetilde{C_E}\le (1+C_P\beta M_0^2)/(1+C_P)$, then $\|z\|_{H^1}^2 \le \widetilde{C_E}^{-1}\|z\|_{a_u}^2$. The final choice of $\widetilde{C_E}$ is the smaller of the two.
\end{proof}

\subsection{Proof of Lemma \ref{lemma:perturb}}
\label{sec:proof_lemma:perturb}
\begin{proof}
    For notational simplicity, we allow the constants $C,\,C'$ to change their meanings through the proof. We also denote
    \begin{align*}
        t:= \|u-v\|_{H^1}.
    \end{align*}
    Using the variational form of the eigenvalues, we have 
    \begin{align*}
        \mu_1 &= \min_{\substack{z\in H_0^1(\Omega), \\ \|z\|_{L^2} = 1}} (z,z)_{a_u} \le (v,v)_{a_u}, \\
        \lambda_1 &= \min_{\substack{z\in H_0^1(\Omega), \\ \|z\|_{L^2} = 1}} (z,z)_{a_v} \le (w_1,w_1)_{a_v}, \\
        \lambda_1 + \lambda_2 &= \min_{\substack{z_1, z_2 \in H_0^1(\Omega), \\ \|z_1\|_{L^2} = \|z_2\|_{L^2} = 1, \\  z_1 \perp z_2 }} (z_1,z_1)_{a_v} + (z_2,z_2)_{a_v} \le (w_1, w_1)_{a_v} + (w_2, w_2)_{a_v}.
    \end{align*}
    We will use the above relations to bound the gap between $\mu_1$ and $\lambda_1$, and $\lambda_2$ and $\mu_2$. First, we have
    \begin{align*}
        \mu_1 \le (v,v)_{a_u} &= (v,v)_{a_v} + \int_{\Omega} \beta (u^2v^2-v^4) \\
        &= \lambda_1 + \int_{\Omega} \beta v^2(u+v)(u-v) \\
        &\le  \lambda_1 + 2\beta M_0^3 \int_{\Omega}|u-v| \\
        &\le \lambda_1 + C(\beta,M_0,\Omega) \cdot t.
    \end{align*}
    Therefore, there exists $C=C(\beta,M_0,\Omega)$ such that when $t\le C$, 
    \begin{align}
    \label{eq:tmp0}
        \mu_1 \le \lambda_1 + \frac{1}{6} C_v.
    \end{align}
    Next, we note that 
    \begin{align}
    \label{eq:tmp1}
    \begin{split}
        \lambda_1 + \lambda_2 &\le (w_1, w_1)_{a_v} + (w_2, w_2)_{a_v} \\
        &= (w_1, w_1)_{a_u} + (w_2, w_2)_{a_u} + \int_{\Omega}\beta(v^2-u^2)(w_1^2+w_2^2) \\
        &= \mu_1 + \mu_2 + \int_{\Omega}\beta (v+u)(v-u)(w_1^2+w_2^2).
    \end{split}
    \end{align}
    To estimate $\|w_1\|_{L^\infty}$, note that it is the weak solution of 
    \begin{align*}
        -\Delta w_1 + V w_1 + \beta u^2 w_1 = \mu_1 w_1.
    \end{align*}
    Since $V, \, u\in L^\infty(\Omega)$, by elliptic regularity, we get
    \begin{align*}
        \|w_1\|_{H^2} &\le C(\beta,V,M_0,\Omega) (\|w_1\|_{H^1} + \mu_1\|w_1\|_{L^2}) \\
        &\le C(\beta,V,M_0,\Omega) + C'(\beta,V,M_0,\Omega) \cdot \mu_1.
    \end{align*}
    When $d\le 3$, using Sobolev embedding, we obtain
    \begin{align*}
        \|w_1\|_{L^\infty} \le C(\Omega) \|w_1\|_{H^2}.
    \end{align*}
    Since we have shown that $\mu_1 \le \lambda_1 + C\cdot t$, putting them together we have
    \begin{align*}
        \|w_1\|_{L^\infty} \le C(\beta,V,M_0,\Omega,\lambda_1) + C'(\beta,V,M_0,\Omega,\lambda_1)\cdot t.
    \end{align*}
    Similarly, we can prove that\footnote{We omit the details of showing  $\mu_2 \le \lambda_2 + C\cdot t$ by showing $\mu_1 + \mu_2 \le \lambda_1 + \lambda_2 + C\cdot t$ using the variational form.}
    \begin{align*}
        \|w_2\|_{L^\infty} \le C(\beta,V,M_0,\Omega,\lambda_1,\lambda_2) + C'(\beta,V,M_0,\Omega,\lambda_1,\lambda_2)\cdot t.
    \end{align*}
    Plugging them back into (\ref{eq:tmp1}), we have
    \begin{align*}
        (w_1, w_1)_{a_v} + (w_2, w_2)_{a_v} &\le \mu_1+\mu_2 + \left(C(\beta,V,M_0,\Omega,\lambda_1,\lambda_2) + C'(\beta,V,M_0,\Omega,\lambda_1,\lambda_2)\cdot t\right)^2 \cdot t.
    \end{align*}
    Therefore, there exists $C= C(\beta,V,M_0,\Omega,\lambda_1,\lambda_2)$, such that when $t \le C$,
    \begin{align}
    \label{eq:tmp2}
        \lambda_1 + \lambda_2 \le \mu_1 + \mu_2 + \frac{1}{6}C_v.
    \end{align}
    Combining (\ref{eq:tmp0}) and (\ref{eq:tmp2}), we have 
    \begin{align}
        \label{eq:tmp3}
        \mu_1 \le \lambda_1 +\frac{1}{6}C_v, \qquad \mu_2 \ge \lambda_2 - \frac{1}{3}C_v, \qquad \mu_2 - \mu_1 \ge \frac{1}{2} C_v.
    \end{align}
    Next, note that
    \begin{align*}
        \lambda_1 \le (w_1,w_1)_{a_v} &= (w_1,w_1)_{a_u} + \int_{\Omega} \beta (v^2-u^2)w_1^2 \\
        &\le \mu_1 + C(\beta,V,M_0,\Omega)\|w_0\|_{L^\infty}^2 \cdot t\\
        &\le \mu_1 + ( C(\beta,V,M_0,\Omega)+C'(\beta,V,M_0,\Omega)\cdot t)^2 \cdot t.
    \end{align*}
    Therefore, there exists $C=C(\beta,V,M_0,\Omega,\lambda_1)$ such that when $t\le C$,
    \begin{align}
    \label{eq:tmp4}
        \lambda_1 \le \mu_1 + \frac{1}{6}C_v.
    \end{align}
    Equations (\ref{eq:tmp0}), (\ref{eq:tmp3}) and (\ref{eq:tmp4}) contain all the relations between $\lambda_1,\,\lambda_2,\,\mu_1$, and $\mu_2$ that we will need.
    
    Since $\{w_i\}_{i=1}^\infty$ forms an orthonormal basis of $H_0^1(\Omega)$, in order to estimate $\|u-w_1\|_{L^2}$, it suffices to bound $(u,u)_{a_u} - \mu_1$. Note that 
    \begin{align*}
        (u,u)_{a_u} - \lambda_1 &= (u,u)_{a_u} - (v,v)_{a_v}  \\
        &= (u,u)_{a_u} - (v,v)_{a_u} + \int_{\Omega} \beta (u^2v^2-v^4) \\
        &\le (\|u\|_{a_u}+\|v\|_{a_u})\cdot\|u-v\|_{a_u} + \int_{\Omega} \beta v^2 (u+v)(u-v) \\
        &\le C(\beta, V, M_0, \Omega) (\|u\|_{H^1}+\|v\|_{H^1})\cdot\|u-v\|_{H^1} + \int_\Omega \beta v^2 (u+v)(u-v) \\
        &\le C(\beta, V, M_0, \Omega)\cdot t.
    \end{align*}
    The fourth inequality uses the norm equivalence in Lemma \ref{lemma:equiv}.
    Thus, there exists $C=C(\beta, V, M_0, \Omega)$, such that when $t\le C$, 
    \begin{align}
        \label{eq:tmp5}
        (u,u)_{a_u} - \lambda_1 \le \frac{1}{12} C_v.
    \end{align}
    Combining (\ref{eq:tmp3}), (\ref{eq:tmp4}) and (\ref{eq:tmp5}), we have
    \begin{align*}
        (u,u)_{a_u} - \mu_1 \le \frac{1}{4}C_v \le \frac{1}{2}(\mu_2-\mu_1).
    \end{align*}
    Assume that $u =\sum_{i=1}^\infty c_i w_i$, where $\sum_{i=1}^\infty c_i^2 = 1$. Then we get
    \begin{align*}
        (u,u)_{a_u} - \mu_1 = \sum_{i=1}^\infty c_i^2 \mu_i - \mu_1 \ge c_1^2\mu_1 + \sum_{i=2}^\infty c_i^2\mu_2 -\mu_1 = (1-c_1^2) (\mu_2-\mu_1).
    \end{align*}
    Since $(u,u)_{a_u} - \mu_1 \le \frac{1}{2}(\mu_2-\mu_1)$, we have
    \begin{align*}
        1-c_1^2 \le \frac{1}{2}, \qquad |c_1| \ge \frac{1}{\sqrt{2}}.
    \end{align*}
    If $c_1 \le -1/\sqrt{2}$, we can use $-w_1$ to replace $w_1$. Thus, we always have $c_1 \ge 1/\sqrt{2}$. This gives
    \begin{align*}
        \|u-w_1\|_{L^2} = \sqrt{2-2c_1} \le \sqrt{2-\sqrt{2}} < 1.
    \end{align*}
    In other words, $s \le \sqrt{2-\sqrt{2}}$. The constant $C$ in the statement of the lemma is the smallest of all the constants $C$, $C'$ in the proof. Since $\lambda_2 = \lambda_1 + C_v$, the dependence on $\lambda_2$ is the dependence on $C_v$.
\end{proof}

\subsection{Proof of Lemma \ref{lemma:linear}}
\label{sec:proof_lemma:linear}
\begin{proof}
    Since $\mu_2$ is strictly greater than $\mu_1$, we can split $\A$ and $u$ as
    \begin{align*}
        &\A = \A^{(1)}+\A^{(2)}, \quad \A^{(1)} =  \A P_{w_1}, \quad \A^{(2)} = \A P_{w_1}^\perp, \\
        &u = u^{(1)} + u^{(2)}, \quad u^{(1)} = P_{w_1} u, \quad u^{(2)} = P_{w_1}^\perp u.
    \end{align*}
    Here $P_{w_1}$ is the orthogonal projection onto the subspace of $w_1$ under the $L^2$ inner product, and $P_{w_1}^\perp = id - P_{w_1}$. 
    Then $\A^{(1)} u^{(1)} = \mu_1 u^{(1)}$, and $(u^{(2)}, u^{(2)})_{\A^{(2)}} \ge \mu_2 \|u^{(2)}\|_{L^2}^2$ since $u^{(2)} \perp w_1$. 
    By definition of $\G$, $(u, \G v)_\A = (u, v)_{L^2}$ for any $u, v \in X$. We have
    \begin{align*}
        (u, \G u^{(1)})_{L^2} &= \mu_1^{-1} \|u^{(1)}\|^2_{L^2}, \\ 
        (u, \G u^{(2)})_{L^2} &= (u^{(1)}, \G u^{(2)})_{L^2} + (u^{(2)}, \G u^{(2)})_{L^2} = (u^{(2)}, \G u^{(2)})_{L^2} ,\\
        (u^{(2)}, \G u^{(2)})_{L^2} &= (\G u^{(2)}, \G u^{(2)})_{\A} \ge \mu_2 \|\G u^{(2)}\|^2_{L^2} \\
        &= \mu_2 \|u^{(2)}\|^{-2}_{L^2} \cdot (\|\G u^{(2)}\|^2_{L^2} \|u^{(2)}\|^2_{L^2}) \ge  \mu_2 \|u^{(2)}\|^{-2}_{L^2} \cdot (u^{(2)}, \G u^{(2)})^2_{L^2}, \\
        \text{i.e., }& (u, \G u^{(2)})_{L^2} \le \mu_2^{-1} \|u^{(2)}\|_{L^2}^2.
    \end{align*}
    Therefore, the objective inequality is transformed into
    \begin{align*}
        & C_L \left((u,u)_\A - \frac{1}{(u,\G u)_{L^2}} \right) - \left((u,u)_\A - (w_1, w_1)_\A \right)\\
        &= (C_L-1) (u,u)_\A - \frac{C_L}{(u,\G u)_{L^2}} + \mu_1   \\
        &= (C_L-1) ((u^{(1)}, u^{(1)})_{\A^{(1)}} + (u^{(2)}, u^{(2)})_{\A^{(2)}}) - \frac{C_L}{(u, \G u^{(1)})_{L^2} + (u, \G u^{(2)})_{L^2}} + \mu_1  \\
        & \ge (C_L-1) (\mu_1\|u^{(1)}\|^2_{L^2} + \mu_2 \|u^{(2)}\|^2_{L^2} ) - \frac{C_L}{\mu_1^{-1} \|u^{(1)}\|^2_{L^2} + \mu_2^{-1} \|u^{(2)}\|_{L^2}^2} + \mu_1\\
        &= (C_L-1) (\mu_1 + (\mu_2-\mu_1) \|u^{(2)}\|^2_{L^2} ) - \frac{C_L \mu_1 \mu_2}{\mu_2 + (\mu_1-\mu_2) \|u^{(2)}\|_{L^2}^2} + \mu_1 \\
        &= (\mu_2-\mu_1)\frac{((C_L-1)\mu_2-C_L\mu_1) \|u^{(2)}\|_{L^2}^2 - (C_L-1) (\mu_2-\mu_1) \|u^{(2)}\|_{L^2}^4}{\mu_2 + (\mu_1-\mu_2) \|u^{(2)}\|_{L^2}^2}.
    \end{align*}
     We look for $C_L$ and $u$ such that the above is greater than or equal to 0. In fact, for any $C_L>1$, if
     \begin{align*}
         0 \le \|u^{(2)}\|_{L^2}^2  \le \frac{(C_L-1)\mu_2-C_L\mu_1}{(C_L-1) (\mu_2-\mu_1)},
     \end{align*}
     then this is satisfied. Note that $\|u-v_1\|_{L^2} \le {s}$ implies $\|u^{(2)}\|_{L^2}^2 \le {s}^2$. So the requirement on $C_L$ is 
     \begin{align*}
         C_L \ge 1+\frac{\mu_2}{(\mu_2-\mu_1)(1-{s}^2)}.
     \end{align*}
\end{proof}

\subsection{Proof of Lemma \ref{lemma:HOI}}
\label{sec:proof_lemma:HOI}
\begin{proof}
    The main idea of the proof is the same as that of Lemma \ref{lemma:perturb} so we only point out their differences here. For example, to estimate $\mu_1 - \lambda_1$, we have
    \begin{align*}
        \mu_1 \le (v,v)_{a_u} & = (v,v)_{a_v} + \int_{\Omega}  \beta (u^2v^2-v^4) + \int_{\Omega} \delta \left((\nabla(uv)^2 - \nabla(v^2)^2 ) \right) \\
        &= \lambda_1 + \int_\Omega \beta v^2(u+v)(u-v) + \int_\Omega \delta (\nabla (uv) + \nabla (v^2)) (\nabla (uv) - \nabla (v^2)).
    \end{align*}
    The second term is bounded in the same way as the proof of Lemma \ref{lemma:perturb}. Only the third term containing high-order interaction is new. To bound the third term, we note that
    \begin{align*}
        &\int_\Omega \delta (\nabla (uv) + \nabla (v^2)) (\nabla (uv) - \nabla (v^2)) \\
        &= \delta \int_\Omega (v \nabla u  + u \nabla v + 2 v\nabla v)(v \nabla u  + u \nabla v - 2 v\nabla v) \\
        &\le  4 \delta M_0 M_1 \int_\Omega\left|v \nabla u  + u \nabla v - 2 v\nabla v\right| \\
        &=  4 \delta M_0 M_1 \int_\Omega\left|v ( \nabla u-\nabla v) + (u - v)\nabla v\right| \\
        &\le C(\delta,M_0,M_1,\Omega)\|u-v\|_{H^1}.
    \end{align*}
    Similar bounds can be obtained in the estimation of $(\lambda_1+\lambda_2)-(\mu_1+\mu_2)$, $\lambda_1 - \mu_1$, and $(u,u)_{a_u} - \mu_1$. The dependence of the constant $C$ only has two additional dependencies which are $\delta$ and $M_1$.
\end{proof}


\end{document}

%% file: shared.tex
\usepackage[utf8]{inputenc}
\usepackage[centering,top=1.5in,bottom=1.2in,left=1.4in,right=1.4in]{geometry}
\usepackage{verbatim}
\usepackage{graphicx,subcaption}
\usepackage{graphbox}
\usepackage{amsmath,amssymb,amsthm}
\usepackage{hyperref}
\usepackage{bm}
\usepackage{enumitem}
\usepackage[usenames]{color}
\usepackage[dvipsnames]{xcolor}
\newtheoremstyle{myThm} 
    {\topsep}                    
    {\topsep}                    
    {\itshape}                   
    {}                           
    {\bfseries} 
    {.}                          
    {.5em}                       
    {}  

\newtheoremstyle{myRem} 
    {\topsep}                    
    {\topsep}                    
    {}                   
    {}                           
    {\bfseries} 
    {.}                          
    {.5em}                       
    {}  

\newtheoremstyle{myDef} 
    {\topsep}                    
    {\topsep}                    
    {}                   
    {}                           
    {\bfseries} 
    {.}                          
    {.5em}                       
    {}  

\theoremstyle{myThm}
\newtheorem{theorem}{Theorem}[section]
\newtheorem{lemma}[theorem]{Lemma}
\newtheorem{proposition}[theorem]{Proposition}
\newtheorem{corollary}[theorem]{Corollary}

\newtheorem{assumption}[theorem]{Assumptions}

\theoremstyle{myRem}
\newtheorem{remark}[theorem]{Remark}

\theoremstyle{myDef}

\def\R{\mathbb{R}}
\def\C{\mathbb{C}}
\def\M{\mathcal{M}}
\def\A{{\mathcal{A}}}
\def\G{\mathcal{G}}
\def\Gu{\mathcal{G}_u}
\def\Gun{\mathcal{G}_{u_n}}
\def\Lap{{\mathcal{L}}}
\def\O{{\mathcal{O}}}
\def\T{{\mathcal{T}}}

\newcommand{\rd}{\mathrm{d}}

\graphicspath{{./figures/}}

\newcommand{\zzy}[1]{{#1}}

\newcommand{\rev}[1]{{{#1}}}